\documentclass[preprint,12pt]{elsarticle}
\usepackage{amsmath}
\usepackage{amsfonts}
\usepackage{amssymb}
\usepackage{latexsym}
\usepackage{graphicx}
\usepackage{extarrows}
\usepackage{amsmath,amsfonts,amssymb,amsthm,amsxtra}
\usepackage{latexsym}
\usepackage{hyperref}
\usepackage{colortbl,dcolumn}
\usepackage{graphicx}
\usepackage{epstopdf}   
\usepackage{stfloats}
\usepackage{subfigure} 
\usepackage{xmpmulti}  
\usepackage{booktabs}
\usepackage{mathrsfs} 
\usepackage{multirow}
\usepackage{siunitx}
\usepackage{float}

\usepackage[ruled,linesnumbered]{algorithm2e}

\newtheorem{theorem}{Theorem}[section]
\newtheorem{proposition}{Proposition}

\theoremstyle{definition}
\newtheorem{definition}{Definition}
\newtheorem{example}{Example}

\theoremstyle{remark}
\newcommand{\myfont}{\fontsize{5pt}{\baselineskip}\selectfont}
\begin{document}
\begin{frontmatter}

\title{\bf The sparse representation related with fractional heat equations}
\author[a]{Pengtao Li}

\author[b]{Tao Qian}

\author[c]{Ieng Tak Leong}
\author[d]{Wei Qu\footnote{Corresponding author. Email: quwei2math@qq.com.}}

 \address[a]{School of Mathematics and Statistics, Qingdao University, China. Email: ptli@qdu.edu.cn.}

\address[b]{Macau center for mathematical science, Macau University of
Science and Technology, Macau. Email: tqian@must.edu.mo.}

\address[c]{Department of Mathematics, Faculty of Science and Technology, University of
Macau, Macau. Email: itleong@um.edu.mo.}

\address[d]{Laboratory of Mathematics and Complex Systems (Ministry of Education), School of Mathematical Sciences, Beijing Normal University, China. Email: quwei2math@qq.com.}

\begin{abstract}
This study introduces pre-orthogonal adaptive Fourier decomposition (POAFD) to obtain approximations and numerical solutions to the fractional Laplacian initial value problem and the extension problem of Caffarelli  and Silvestre (generalized Poisson equation).  The method, as the first step, expands the initial data function into a sparse series of the fundamental solutions with fast convergence, and, as the second step, makes use the semigroup or the reproducing kernel property of each of the expanding entries. Experiments show effectiveness and efficiency of the proposed series solutions.
\end{abstract}

\begin{keyword} reproducing kernel Hilbert space; dictionary; sparse representation; approximation to the identity; fractional heat equations
\MSC[2020] 65M80; 41A30; 65N80; 35K08; 35K05
\end{keyword}
%

\bigskip

\end{frontmatter}

\section{Introduction}
For $0<\alpha<1$, the  fractional Laplacian of order $\alpha$, denoted by $(-\Delta )^{\alpha}$, can be defined on functions $f:\mathbb R^{n}\rightarrow \mathbb R$ via the Fourier multiplier given by the formula
 $$\widehat{(-\Delta)^{\alpha}f}(\xi)=|\xi|^{2\alpha}\widehat{f}(\xi),$$
where the Fourier transform $\widehat{f}$ of  $f: \mathbb {R}^{n}\rightarrow \mathbb{R}$ is given by
$$\widehat{f}(\xi)=\int_{\mathbb R^{n}}f(x)e^{-ix\cdot\xi}dx.$$
The fractional Laplacian can be alternatively written as the singular integral operator defined by
$$(-\Delta)^{\alpha}f(x)=c_{n,\alpha}\int_{\mathbb R^{n}}\frac{f(x)-f(y)}{|x-y|^{n+2\alpha}}dy,$$
where the prefactor  $c_{n,\alpha}$ is the constant
$$\frac{4^{\alpha}\Gamma(n/2+\alpha)}{\pi^{n/2}|\Gamma(-\alpha)|}$$
 involving the Gamma function and serving as a normalizing factor.
We refer the reader to \cite{K} for other equivalent definitions of the fractional Laplacian.

The fractional  Laplacian $(-\Delta)^{\alpha}$ plays a significant role in many areas of mathematics, such as  harmonic analysis and PDEs, and is often applied to describe many complicated phenomena via partial differential equations. Anomalous diffusion processes with non-locality in complex media can be well characterized by using fractional-order diffusion equation models, e.g. the following fractional heat equation:
\begin{equation}\label{eq-1.1}
\left\{\begin{aligned}
&\frac{\partial u}{\partial t}(t+x)=(-\Delta)^{\alpha}u(t+x),&\ (x,t)\in\mathbb R^{n+1}_{+};\\
&u(0+x)=f(x),&\ x\in\mathbb R^{n}.
\end{aligned}\right.
\end{equation}
In addition, the fractional  Laplacian has been applied to study a wide class of physical systems and engineering problems,  including  L\'{e}vy  flights,  stochastic  interfaces  and  anomalous  diffusion  problems. By the following extension problem for generalized Poisson equation:
\begin{equation}\label{eq-1.2}
\left\{
\begin{aligned}
\hbox{div}(t^{\sigma}\nabla u)(x)&= 0,\quad &\ (x,t)\in\mathbb R^{n+1}_{+};\\
u &= f, &\ x\in\mathbb R^{n},
\end{aligned}
\right.
\end{equation}
L. Caffarelli and L. Silvestre showed in \cite{Caffarelli1} that any fractional power of the Laplacian can be determined as an operator that maps a Dirichlet boundary condition to a Neumann type condition via an extension problem. In the study on the obstacle problem for the
fractional Laplace operators, this characterization of $(-\Delta)^{\alpha}$ via the above local (degenerate) PDE  was firstly used  in \cite{Caffarelli5} to  get related estimates of regularities. We also refer the reader to \cite{Caffarelli6, Caffarelli7, He, RQ-S} for further information on applications of the fractional Laplacian in PDEs.

The aim of this article is to develop sparse decompositions of the solutions to the initial problems given in (\ref{eq-1.1}) and (\ref{eq-1.2}) by using the POAFD (pre-orthogonal adaptive Fourier decomposition) method under the $\mathcal{H}-H_{K}$ formulation of T. Qian (\cite{Q1}) and more precisely the convolution-type sparse representation of the identity which was further developed by  W. Qu et al. in \cite{Q2}.
  We also refer the reader to \cite{Q3, Q4} for the closely related AFD adaptive Fourier decompositions methods. In \cite{Q2} the POAFD methodology is  used to give fast approximation to the identity, and hence to expand the initial
(boundary) data by using the dictionaries elements the parameterized fractional
heat and the Poisson kernels.
Then by the $\lq\lq$lifting up" method based on semigroup properties in the two cases we obtain sparse representation of the original Dirichelet boundary and Cauchy initial value problems.

It should be pointed out that the sparse representations of the solutions to the equation (\ref{eq-1.1}) and (\ref{eq-1.2}) are not simple analogues of those related with the classical heat kernel and Poisson kernels. In \cite{Q2}, an important mechanism together with the sparse representation of the Dirac-$\delta$ generalized function is the following superimposed effect of the classical heat kernel and  Poisson kernels. Precisely, for any $(x,t), (y,s)\in\mathbb R^{n}$,
\begin{equation}\label{eq-1.3}
\left\{\begin{aligned}
\int_{\mathbb R^{n}}\frac{t}{(t^{2}+|x-z|)^{(n+1)/2}}\frac{s}{(s^{2}+|y-z|)^{(n+1)/2}}dz&=\frac{t+s}{((t+s)^{2}+|x-y|)^{(n+1)/2}};\\
\frac{1}{(4\pi)^{n}(ts)^{n/2}}\int_{\mathbb R^{n}}e^{-|x-z|^{2}/4t}e^{-|y-z|^{2}/4s}dz&=\frac{1}{(4\pi(t+s))^{n/2}}e^{-|x-y|^{2}/4(t+s)},
\end{aligned}\right.
\end{equation}
which can be deduced from the uniqueness of the solutions of the heat equation and the Poisson equation. Like the heat kernel and the Poisson kernel cases, the integral kernel of the fractional heat equation (\ref{eq-1.1}), denoted as $K_{\alpha, t+x}(\cdot),$ although without an explicit formula, can also adopt the same argument based on the uniqueness of the solution to (\ref{eq-1.1}) together with the linearity of the $t$-differentiation. There holds
$$\langle K_{\alpha, t+x}(\cdot),\ K_{\alpha, s+y}(\cdot)\rangle_{L^{2}} =K_{\alpha, t+s}(x-y).$$
In Section \ref{sec-3}, we, alternatively, apply the Fourier  multipliers and the Plancherel formula to establish the above semigroup identities for the fractional heat equation (\ref{eq-1.1}):
For the special cases $\alpha=1/2,1,$ the identities become those in (\ref{eq-1.3}) in the given order, respectively.

For the case of equation (\ref{eq-1.2}), the solutions  $\{P^{\sigma}_{t}f\}_{t>0}$ do  not satisfy the semigroup property, or in other words, the analogous relations like those in (\ref{eq-1.3}) are invalid.

Another difference between our results and those obtained in \cite{Q2} is the technology used in proving the boundary vanishing condition (BVC) of the reproducing kernels. In \cite{Q2}, the authors use (\ref{eq-1.3}) to verify the BVC properties of the heat and the Poisson type kernels,  see \cite[Theorems 3.1\ \&\ 3.2 ]{Q2}.  Since the semigroup property does not hold for the fundamental solutions of (\ref{eq-1.2}), we utilize  the fractional Poisson kernel $$p^{\sigma}_{t}(x)=\frac{t^{\sigma}}{(t^{2}+|x-y|^{2})^{(n+2)/\sigma}}.$$
 Through providing a decay estimate of the integral
 $$\frac{t^{n/2}}{(c_{n,\sigma})^{1/2}}\int_{\mathbb R^{n}}\frac{t^{\sigma}}{(t^{2}+|x-z|^{2})^{(n+\sigma)/2}}\frac{s^{\sigma}}{(s^{2}+|y-z|^{2})^{(n+\sigma)/2}}dz,$$
we show the BVC property, see Theorem \ref{thm-4.2}.

 Moreover, Theorem \ref{thm-4.2} partly improves the results of the sparse approximation for the general convolution cases as given in \cite[Theorem 3.3]{Q2}. In the latter the authors consider the sparse approximation for following convolution kernel $\phi_{t}(x):=t^{-n/2}\phi(x/t)$, where
 $$\left\{\begin{aligned}
 &\int_{\mathbb R^{n}}\phi(x)dx=1;\\
 &\sup_{|y|\geq |x|}|\phi(y)|\leq \frac{C}{(1+|x|^{2})^{(n+\delta)/2}},\ \delta>0.
 \end{aligned}\right.$$
 For $\delta\geq 1$, the functions $\phi(\cdot)$ are dominated by the Poisson kernel. Then, as a consequence of \cite[Theorem 3.1]{Q2}, the sparse representation in $\{\phi_{t}(\cdot)\}_{t>0}$  can be established. Our treatment in Theorem \ref{thm-4.2}, however, can guarantee the sparse representation for all $\delta>0$.
\def\E{\mathcal E}

{\it Some notations}: Throughout this paper, we use the symbol
 $U\simeq V$ to denote there is a constant $c>0$ such that $c^{-1}V\leq U\leq cV$. The symbol $U\lesssim V$ means that $U\leq cV$. Similarly, we use $V\gtrsim U$ to denote there exists a constant $c$ such that $V\geq cU$.

\section{Preliminary}
Firstly, we state some preliminaries on the $\mathcal{H}$-$H_{K}$ formulation, which was introduced by Qian in \cite{Q1} based on the theory of reproducing Hilbert spaces. Let $\mathcal{H}$ be a Hilbert space whose inner product is denoted by $\langle\cdot,\cdot\rangle_{\mathcal{H}}$. Correspondingly, for any $f\in\mathcal H$, the norm $\|f\|_{\mathcal H}^{2}=\langle f,\ f\rangle_{\mathcal H}$.
Let $\E$ be an open  set in the underline topological space, where, in this article, the elements are real or complex numbers, or vectors.
In the $\mathcal{H}$-$H_{K}$ formulation, each $p\in\E$ is treated as as a parameter. Precisely, for $p\in\E$, there exists an  element $h_{p}\in\mathcal H$ corresponding to $p$. Denote by $\mathbb{C}^{\E}$ the set of all functions from $\E$ to the complex number field $\mathbb{C}$. For  $f\in\mathcal{H}$, we can define a linear operator from $\mathcal{H}$ to $\mathbb{C}^{\E}$, denoted by $\mathscr{L}$ via the inner product $\langle\cdot,\ \cdot\rangle$ on $\mathcal H$, i.e.,
\begin{equation}
\mathscr{L}(f)(p)=\langle f,\ h_{p}\rangle_{\mathcal{H}},\ p\in \E.
\end{equation}
We use $\mathcal{N}(\mathscr{L})$ to denote the null space of $\mathscr{L}$, i.e., the set of all elements $f\in\mathcal{H}$ such that $\mathscr{L}(f)=0$.
Take a sequence of functions $\{f_{n}\}_{n\in\mathbb N}$ in  $N(\mathscr L)\subset \mathcal{H}$ satisfying
$$\lim_{n\rightarrow\infty}\|f-f_{n}\|_{\mathcal H}=0.$$
Since
\begin{eqnarray*}
|\mathscr{L}(f)(p)|&=&|\mathscr{L}(f-f_{n})(p)|\\
&=&|\langle f-f_{n},\ h_{p}\rangle_{\mathcal{H}}|\\
&\leq&\|f-f_{n}\|_{\mathcal{H}}\|h_{p}\|_{\mathcal{H}},
\end{eqnarray*}
letting $n\rightarrow \infty$ reaches to $|\mathscr{L}(f)(p)|=0,$ which implies that $f\in N(\mathscr L)$ and hence, we can see that $N(\mathscr L)\subset\mathcal{H} $ is closed. Denote by
$N(\mathscr{L})^{\perp}$ denote  the trivial or non-trivial orthogonal complement of $N(\mathscr{L})$ in $\mathcal{H}$.
Then
$$\mathcal{H}=N(\mathscr{L})\oplus N(\mathscr{L})^{\perp}.$$
 For any $f\in\mathcal H$, there exist
unique $f_{-}\in N(\mathscr L)$ and unique $f_{+}\in N(\mathscr{L})^{\perp}$ such that
$$f=f_{-}+f_{+}.$$ Let $F(p)=\mathscr L(f)(p)$ and denote by $R(\mathscr{L})$ the rang of $\mathscr L$:
$$R(\mathscr L):=\Big\{F\in \mathbb{C}^{\E}:\ \text{ there exists} f\in\mathcal{H}\text{ such that } F=\mathscr{L}(f)\Big\}.$$
Obviously, for any $f\in\mathcal H$, it holds
$$\mathscr L(f)=\mathscr L(f_{-})+\mathscr L(f_{+})=\mathscr L(f_{+}),$$
which implies that $\mathscr L(N(\mathscr L)^{\perp})=R(\mathscr L)$.
Define the space $H_{K}$ as the set of all $F\in R(\mathscr L)$ satisfying $\|F\|_{H_{K}}<\infty$, where
\begin{equation}\label{consequence}
\|F\|_{H_{K}}:=\|f_{+}\|_{\mathcal{H}},\quad F=\mathscr L(f)\ \&\ f\in\mathcal{H},
\end{equation}
and denote by $\langle\cdot,\cdot\rangle_{H_{K}}$ the inner product generated by $\|\cdot\|_{H_{K}}$ via the polarization identity.
In \cite{Q1}, Qian proved the following result.
\begin{proposition}{\rm (\cite[page 3]{Q1})}
\item{\rm (i)} The space $H_{K}$ is a Hilbert space with the inner product $\langle\cdot,\cdot\rangle_{H_{K}}$.
\item{\rm (ii)} Via the mapping $\mathscr L$, the space $H_{K}$ is isometric with $N(\mathscr L)^{\perp}$.
\item{\rm (iii)} Define
$$K(q,p):=\langle h_{q}, h_{p}\rangle_{\mathcal{H}}.$$
The function $K(q,p)$ is the reproducing kernel of $H_{K}$.
\end{proposition}

When studying  linear operators in Hilbert spaces, the $\mathcal{H}$-$H_{K}$ formulation provides a greatly facilitates. Especially, in practice, this formulation is applied to offer fast converging numerical solutions, see \cite{Q1, Q2}.  In each of our two contexts, $\E$ is ${\mathbb R^{n+1}_+}$,  $\mathcal{H}$ is $L^2({\mathbb R^{n}}).$ The kernel functions $h_p$ are the respective integral kernels giving rise to the solutions, the latter being in the image Hilbert spaces which are the reproducing kernel spaces $H_K.$
Since  the dictionary generated by the parameterized kernels $h_p$ is dense,  $H_K$ is isometric with $L^2({\mathbb R^{n}})$ in each of our two cases.  In the case there holds, for $u=\mathscr Lf$, noting that the reproducing kernel of $H_{K}$ is $K_{p}$, then
$u(p)=\langle u,\ K_{p}\rangle_{H_{K}},$
which together with the definition of $\mathscr{L}$, indicates that
$$\langle u,\ K_{p}\rangle_{H_{K}}=u(p)=\mathscr{L}(f)(p)=\langle f,\ h_{p}\rangle_{L^{2}(\mathbb R^{n})}.$$

Now we state briefly our algorithm, the pre-orthogonal adaptive Fourier decomposition (POAFD), used in the article. Fundamentally, the POAFD  can be seen as a modified greed algorithm of sparse representations. The main idea of the POAFD  algorithm combines the features and advantages of both the greedy algorithm and the  adaptive
Fourier decomposition (AFD). On the one hand,  like the greedy principle, the POAFD algorithm makes the locally optimal choice at each step to  seek finding a global optimum. On the other hand, similar to the AFD decomposition which induces the rational approximation via  Blaschke products with multiple zeros, in order to  construct the approximation, the POAFD algorithm selects  multiple kernels with the same parameters in the process of Gram-Schmidt orthogonalization.

\begin{definition}\label{def-2.2}
 Let $H_{K}$ be a reproducing kernel Hilbert space with the reproducing kernel $K_{q},\ q\in\mathcal{E}$. In the $\mathcal{H}-H_{K}$ formulation,
a dictionary can be constituted by the  normalization of $\{K_{q}\}_{q\in\E}$, denoted by $\{E_{q}\}_{q\in\E}$, which is defined as
$$E_{q}:=\frac{K_{q}}{\|K_{q}\|_{H_{K}}}.$$
The so called POAFD method is to find, consecutively, a sequence $q_1,\ldots,q_k,\ldots$ in $\E$ so that
$$ q_k=s_k+y_k=\arg\sup_{q\in \E} \Big\{|\langle f,E_k^q\rangle|\Big\},$$
where $(E_1,\ldots,E_{k-1},E_k^q)$ is the Gram-Schmidt process of the kernels $(\widetilde{K}_{q_1}, \ldots, \widetilde{K}_{q_{k-1}}, \widetilde{K}_q),$ in which the Gram-Schmidt orthogonalization of $(\widetilde{K}_{q_1},\ldots,\widetilde{K}_{q_{k-1}})$ is denoted by $(E_1,\ldots,E_{k-1})$.
\end{definition}

The POAFD type sparse representation of the initial or boundary data $f$ is, with a sequence of suitable constants $\{c_n\}$,
$$f=\sum_{k=1}^\infty \langle f,E_k\rangle E_k=\sum_{k=1}^\infty c_k \widetilde{K}_{q_k}.$$
Both the solutions for the problems (\ref{eq-1.1}) and (\ref{eq-1.2}) are of the same form
\[u(p)=\langle f,K_p\rangle=\sum_{k=1}^\infty c_k \langle\widetilde{K}_{q_k},K_p\rangle=
\sum_{k=1}^\infty c_k \widetilde{K}_{q_k}(p), \ p=t+x.\]
In the (\ref{eq-1.1}) case, owing to the semigroup property the solution may be rewritten as
\[ u(t+x)=\sum_{k=1}^\infty c_k \widetilde{K}_{\alpha, t+s}(x-y_k)\]
(See \S 3). In the (\ref{eq-1.2}) case there is no semigroup property and one has to compute
\begin{eqnarray*}\label{has}
\widetilde{K}_{q_k}(p)=\langle \widetilde{K}_{q_n},K_p\rangle
\end{eqnarray*}
according to (\ref{has to}). We finally note that both the solutions have the same
convergence rate: The error for the $N$-partial sum in the norm sense is $O(\frac{1}{\sqrt{N}})$
(\cite{Q3}).

\section{Sparse representation of fractional heat equations }\label{sec-3}
In this section, we study sparse representation of approximation to the identity via the fractional heat equations (\ref{eq-1.1}). Throughout the rest of this paper, we take the Hilbert space $\mathcal{H}=L^{2}(\mathbb R^{n})$, defined as the set of all complex-valued  measurable functions $f$ over $\mathbb R^{n}$ satisfying
$$\int_{\mathbb R^{n}}|f(x)|^{2}dx<\infty.$$
 The inner product on $L^{2}(\mathbb R^{n})$ is defined via the Lebesgue integral as
$$\langle f,\ g\rangle:=\int_{\mathbb R^{n}}f(x)\overline{g(x)}dx,\quad \forall\ f,g\in L^2({\mathbb R^{n}}).$$
The corresponding norm of $L^{2}(\mathbb R^{n})$ is
$$\|f\|_{L^{2}(\mathbb R^{n})}^{2}=\langle f,\ f\rangle:=\int_{\mathbb R^{n}}|f(x)|^{2}dx.$$
For the initial value problem (\ref{eq-1.1}), the unique solution $u(t+x)$ can be represented by the fractional integral transform
$$u(t+x)=\mathscr L(f)(t+x):=\int_{\mathbb R^{n}}K_{\alpha,t}(x-y)f(y)dy.$$
 We note that the notation $t+x$ has the meaning $(t,(x_1,\ldots,x_n))$ with $x=(x_1,\ldots,x_n)\in \mathbb R^{n}.$  By the Fourier transform, the fractional Laplacian $(-\Delta)^{\alpha}$ can be equivalently represented as,
  for $\xi=(\xi_{1},\xi_{2}, \ldots,\xi_{n})$,
$$\widehat{(-\Delta)^{\alpha}f}(\xi)=|\xi|^{2\alpha}\widehat{f}(\xi).$$
The fractional heat semigroup can be represented as
$$\widehat{(e^{-t(-\Delta)^{\alpha}})f}(\xi)=e^{-t|\xi|^{2\alpha}}\widehat{f}(\xi).$$
Denote by $K_{\alpha,t}(\cdot)$ the integral kernel related with $e^{-t(-\Delta)^{\alpha}}$, i.e.,
$$e^{-t(-\Delta)^{\alpha}}f(x)=\int_{\mathbb R^{n}}K_{\alpha,t}(x-y)f(y)dy.$$
By the inverse Fourier transform, the fractional heat kernel can be represented as
\begin{eqnarray}\label{ref}K_{\alpha,t}(x)=\frac{1}{(2\pi)^{n}}\int_{\mathbb R^{n}}e^{-t|\xi|^{2\alpha}}e^{ix\xi}d\xi.\end{eqnarray}
One has the following basic estimate:
\begin{proposition}\label{prop-1}{\rm (\cite{BG, CS, XZ})}
Under $\alpha\in(0,1)$, the fractional heat kernel satisfies the following estimates
$$K_{\alpha,t}(x)\simeq \frac{t}{(t^{1/(2\alpha)}+|x|)^{n+2\alpha}},\quad \forall\ (x,t)\in\mathbb R^{n+1}_{+}.$$
\end{proposition}

Because $\{e^{-t(-\Delta)^{\alpha}}\}_{t\geq0}$ is a strongly continuous semigroup in $L^{2}(\mathbb R^{n})$. Then
$$\lim_{t\rightarrow 0+}u(t+x)=f(x)$$
in both the sense of $L^{2}(\mathbb R^{n})$  and  the pointwise sense almost everywhere. Based on this fact, we denote $f(x)=u(0+x)$. Due to the isometric relation there hold
\begin{eqnarray}\label{norm e}\|u\|^{2}_{H_{K}}=\|f\|^{2}_{L^{2}(\mathbb R^{n})}=\|u(0+x)\|_{L^{2}(\mathbb R^{n})}^{2}.\end{eqnarray}
To construct the approximation, we apply the Plancherel theorem to deduce the following semigroup property: for $q=t+x$ and $p=s+y$,
\begin{eqnarray*}\label{be}
K(q,p)&=&\langle h_{q},\ h_{p}\rangle_{L^{2}}\nonumber \\
&=&\langle K_{\alpha, t+x}(\cdot),\ K_{\alpha, s+y}(\cdot)\rangle_{L^{2}}\nonumber \\
&=&\int_{\mathbb R^{n}}e^{-(t+s)|\xi|^{2\alpha}}e^{i(x-y)\xi}d\xi\nonumber\\
&=&K_{\alpha, t+s}(x-y).
\end{eqnarray*}
Denote $h_{q}(y)=K_{\alpha,t}(x-y)$, where $q=t+x$, For $t+s>0$, let $q_{1}=t+s+x$ and $q_{2}=t+s+y$. Then
$h_{q_{1}}(y)=K_{\alpha,s+t}(x-y)$ and $h_{q_{2}}(x)=K_{\alpha,s+t}(y-x)$. Hence
$$K(q,p)=h_{t+s+y}(x)=h_{t+s+x}(y)=h_{t+s+x-y}(0).$$
On the other hand, it follows from the Fourier transform and the change of variable: $\xi=-\eta$ that
\begin{eqnarray*}
K_{\alpha, s+t}(x-y)&=&\int_{\mathbb R^{n}}e^{-(t+s)|\xi|^{2\alpha}}e^{i(x-y)\xi}d\xi\\
&=&\int_{\mathbb R^{n}}e^{-(t+s)|\eta|^{2\alpha}}e^{i(y-x)\eta}d\eta\\
&=&K_{\alpha,s+t}(y-x).
\end{eqnarray*}
Now we give the sparse representation theorem related to the fractional Laplace equations.
\begin{theorem}
The dictionary of the fractional heat kernels satisfies BVC. As a consequence, POAFD algorithm can be performed in the context to obtain sparse representation of functions in $\mathcal{H}=L^{2}(\mathbb R^{n}).$
\end{theorem}
\begin{proof}
For $q=t+x$, we can compute the norm   $\|K_{q}\|_{H_{K}}$ as follows.
$$\|K_{q}\|_{H_{K}}^{2}=\langle K_{q},\ K_{q}\rangle_{H_{K}}=h_{2t+x-x}(0)=h_{2t+0}(0).$$
By the Fourier transform,
\begin{eqnarray*}
h_{2t+0}(0)&=&K_{\alpha,2t}(0)\\
&=&\int_{\mathbb R^{n}}e^{-2t|\xi|^{2\alpha}}e^{i0\cdot \xi}d\xi\\
&=&C_n\int^{\infty}_{0}e^{-2t|\xi|^{2\alpha}}|\xi|^{n-1}d|\xi|\\
&=&\frac{C_n}{(2t)^{n/2\alpha}}\int^{\infty}_{0}e^{-u^{2\alpha}}u^{n-1}du\\
&=&C_{n,\alpha}(2t)^{-n/(2\alpha)}.
\end{eqnarray*}
Hence, for $q=t+x\ \&\ p=s+y$, the normalization of the reproducing kernel $K_{q}$  can be expressed as
$$\|K_{q}\|^{2}_{H_{K}}=\langle K_{q},\ K_{q}\rangle_{L^{2}}=\frac{C_{n,\alpha}}{(2t)^{n/(2\alpha)}}.$$
Set
$$E_{q}=\frac{K_{q}}{\|K_{q}\|_{H_{K}}}=\left(\frac{(2t)^{n/(2\alpha)}}{C_{n,\alpha}}\right)^{1/2}K_{q}$$
such that $\|E_{q}\|_{H_{K}}=1$. Then
$$E_{q}(p)=\frac{(2t)^{n/(4\alpha)}}{\sqrt{C_{n,\alpha}}}h_{t+s+x}(y)=
\frac{(2t)^{n/(4\alpha)}}{\sqrt{C_{n,\alpha}}}K_{\alpha,t+s}(x-y).$$
Below we begin to prove BVC, i.e., for $u\in H_{K}$,
$$\lim_{\mathbb{R}^{n+1}_{+}\ni q\rightarrow \partial^{\ast}\mathbb R^{n+1}_{+}}|\langle u,\ E_{q}\rangle_{H_{K}}|=0.$$
Here we use the one point compactification topology in which $\partial^{\ast}\mathbb R^{n+1}_{+}$ stands for the boundary of $\mathbb R^{n+1}_{+}$ in the topology. Since the span of $\{K_p\}_{p\in \mathbb R^{n+1}_{+}}$ is dense in $L^2(\mathbb R^{n}),$ it suffices to verify that for any fixed $p=s+y$,
\begin{equation}\label{eq-3.1}
\lim_{ q\rightarrow \partial^{\ast}\mathbb R^{n+1}_{+}}|\langle K_{p},\ E_{q}\rangle_{H_{K}}|=0.
\end{equation}
Applying Proposition \ref{prop-1},
$$\langle K_{p},\ E_{q}\rangle_{H_{K}}=E_{p,q}=\frac{(2t)^{n/(4\alpha)}}{\sqrt{C_{n,\alpha}}}K_{\alpha, s+t+x}(y),$$
we can see that there exist two constants $c_{1}<c_{2}$ which are independent of $(x, t)\in\mathbb R^{n+1}_{+}$ such that
$$\frac{c_{1}(t+s)}{((t+s)^{1/(2\alpha)}+|x-y|)^{n+2\alpha}}\leq K_{\alpha, s+t+x}(y)\leq \frac{c_{2}(t+s)}{((t+s)^{1/(2\alpha)}+|x-y|)^{n+2\alpha}}.$$
Hence the statement (\ref{eq-3.1}) is equivalent to the following:
\begin{equation}\label{eq-3.2}
\lim_{ q=t+x\rightarrow \partial^{\ast}\mathbb R^{n+1}_{+}}\frac{(2t)^{n/(4\alpha)}}{\sqrt{C_{n,\alpha}}}\frac{(t+s)}{((t+s)^{1/(2\alpha)}+|x-y|)^{n+2\alpha}}=0.
\end{equation}

Now we begin to prove (\ref{eq-3.2}). We can obtain
\begin{eqnarray*}
\frac{(2t)^{n/(4\alpha)}}{\sqrt{C_{n,\alpha}}}\frac{t+s}{((t+s)^{1/(2\alpha)}+|x-y|)^{n+2\alpha}}
&\lesssim&\frac{(2t)^{n/(4\alpha)}}{\sqrt{C_{n,\alpha}}}\frac{t+s}{(t+s)^{(n+2\alpha)/(2\alpha)}}\\
&\lesssim&\left(\frac{t}{t+s}\right)^{n/(4\alpha)},
\end{eqnarray*}
which implies that
$$\lim_{t\rightarrow 0+}|\langle K_{p},\ E_{q}\rangle_{H_{K}}|=0.$$

Now we analyze the process $R\rightarrow \infty$, where $R=\sqrt{t^{1/\alpha}+|x|^{2}}$. We split the rest of the proof into two cases.

{\it Case 1: $t^{1/(2\alpha)}<\frac{1}{2}\sqrt{t^{1/\alpha}+|x|^{2}}=R/2$.} For this case, $|x|>\sqrt{3}t^{1/(2\alpha)}$ and $\sqrt{3}R<|x|$. Without loss of generality, we assume $R>4|y|$, then
$$|x-y|\geq |x|-|y|\geq \sqrt{3}R-R/4.$$
Now we can get
\begin{eqnarray*}
\frac{(2t)^{n/(4\alpha)}}{\sqrt{C_{n,\alpha}}}K_{\alpha, t+s+x}(y)&\lesssim&\left(\frac{2t}{C_{n,\alpha}}\right)^{n/(4\alpha)}\frac{t+s}{((t+s)^{1/(2\alpha)}+|x-y|)^{n+2\alpha}}\\
&\lesssim&(R/2)^{n/2}\frac{t+s}{(t+s)|x-y|^{n}}\\
&\lesssim&\frac{1}{R^{n/2}},
\end{eqnarray*}
which gives
$$\lim_{R\rightarrow \infty}|\langle K_{p},\ E_{q}\rangle_{H_{K}}|=0.$$

{\it Case 2: $t^{1/(2\alpha)}\geq R/2$.} For $s>0$, we can also get
\begin{eqnarray*}
\frac{(2t)^{n/(4\alpha)}}{\sqrt{C_{n,\alpha}}}K_{\alpha, t+s+x}(y)&\lesssim&\frac{(2t)^{n/(4\alpha)}}{\sqrt{C_{n,\alpha}}}\frac{t+s}{((t+s)^{1/(2\alpha)}+|x-y|)^{n+2\alpha}}\\
&\lesssim&\frac{(2t)^{n/(4\alpha)}}{\sqrt{C_{n,\alpha}}}\frac{t+s}{(t+s)^{(n+2\alpha)/(2\alpha)}}\\
&\lesssim&\frac{1}{(t+s)^{n/(4\alpha)}}\\
&\lesssim&\frac{1}{R^{n/(4\alpha)}}.
\end{eqnarray*}
We can also get
$$\lim_{R\rightarrow \infty}|\langle K_{p},\ E_{q}\rangle_{H_{K}}|=0.$$
\end{proof}
By the above theorem, the initial data $f$ of (\ref{eq-1.1}) has a POAFD sparse representation.
\begin{theorem}
Let, under the POAFD scheme, $f$ possess a sparse representation
\[ f(x)=\sum_{k=1}^\infty \langle f, E_k\rangle E_k,\]
where for each $N,$
 the $N$-orthonormal system $(E_1,\ldots,E_N)$ corresponds to the
POAFD maximally selected $(q_1,\ldots,q_n), q_k=s_k+y_k, k=1,\ldots, N.$ Then there exist constants $c_1,\ldots,c_N,\ldots,$ such that
\[ f(x)=\sum_{k=1}^\infty c_k \widetilde{K}_{\alpha,s_k}(y_k-x),\]
and
\[ u_f(t+x)=\sum_{k=1}^\infty c_k \widetilde{K}_{\alpha,t+s_k}(y_k-x)\]
is the solution of (\ref{eq-1.1}).
\end{theorem}

\section{Sparse representation of fractional Poisson equations }
Given a regular function  $f$  on $\mathbb{R}^n$, if  $u$ is a solution to (\ref{eq-1.2}) with the initial data $f$,   $u(x,t)=P_{\alpha} f(x,t)$ is said to be the Caffarelli-Silvestre extension  of $f$ to the upper half-space $\mathbb{R}^{n+1}_{+}$. For a smooth function $f$,
the   Caffarelli-Silvestre extension of $f$ can be expressed as the convolution of $f$ and the generalized (fractional) Poisson kernel defined as
 $$p^{\alpha}_t(x):=\frac{c(n,\alpha)t^\alpha}{(|x|^2+t^2)^{{(n+\alpha)}/{2}}}.$$ Precisely,
$$P_{\alpha} f(x,t)=p^{\alpha}_t\ast f(x,t)
=c(n,\alpha)\int_{\mathbb{R}^n}\frac{t^{\alpha}}{(|x-y|^2+t^2)^{{(n+\alpha)}/{2}}}f(y)dy.$$
Here  $f\ast g$ means the convolution of $f$ and $g,$ and the constant
$$c(n,\alpha)=\frac{\Gamma({(n+\alpha)}/{2})}{\pi^{n/2}\Gamma({\alpha}/{2})}$$ is chosen such that $\int_{\mathbb{R}^n}p^{\alpha}_t(x)dx=1$.

Now we investigate the sparse representation via the generalized Poisson equations. The fractional Poisson kernel satisfies the following estimates.
\begin{proposition}\label{prop-4.1}{\rm (\cite[Proposition 2.1]{LHZ})}
\item{\rm (i)} There holds
\begin{eqnarray*}
\widehat{p^{\sigma}_{t}}(\xi)&=&\frac{c(n,\sigma)}{\Gamma((n+\sigma)/2)}\int^{\infty}_{0}\lambda^{\sigma/2}
e^{-\lambda-|t\xi|^{2}/(4\lambda)}\frac{d\lambda}{\lambda}:=c(n,s)G_{\sigma}(t|\xi|),
\end{eqnarray*}
where
$$G_{\sigma}(\xi)=\int_{\mathbb R^{n}}e^{-2\pi i\xi\cdot x}\frac{dx}{(1+|x|^{2})^{(n+\sigma)/2}}.$$
\item{\rm (ii)} For some positive constant $c$,
$$G_\sigma(\xi)=\left\{\begin{aligned}
1+|\xi|^\sigma,&\ \xi \text{ near the origin};\\
O(e^{-c\xi}),&\ |\xi|\rightarrow \infty.
\end{aligned}\right.$$
\end{proposition}
By the change of variable, it is easy to verify that
\begin{eqnarray*}
\int_{\mathbb R^{n}}p^{\sigma}_{t}(x)dx&=&c(n,\sigma)\int_{\mathbb R^{n}}\frac{t^{\sigma}}{(t^{2}+|x|^{2})^{(n+\sigma)/2}}dx\\
&=&c(n,\sigma)\int^{\infty}_{0}\frac{1}{(1+u^{2})^{(n+\sigma)/2}}u^{n-1}du\\
&=&1.
\end{eqnarray*}

Let $\mathcal{H}=L^{2}(\mathbb R^{n})$ and $E=\mathbb R^{n+1}_{+}$. For $t,s>0$ and $x,y\in\mathbb R^{n}$, set $q=t+x$ and $p=s+y$. Define $h_{p}(y)$ as
\begin{eqnarray}\label{givenby}
h_{p}(y)=\frac{c(n,\sigma)t^{\sigma}}{(t^{2}+|x-y|^{2})^{(n+\sigma)/2}}.\end{eqnarray}
The space $H_{K}$ is defined as
$$H_{K}:=\Big\{u:\mathbb{R}^{n+1}_{+}\rightarrow\mathbb R,\ u(p)=\langle f,\ h_{p}\rangle_{L^{2}(\mathbb R^{n})}\Big\}.$$
The reproducing kernel is computed
\begin{eqnarray}\label{has to}
K_{q}(p)&=&K(q,p)\nonumber \\
&=&\langle h_{q},\ h_{p}\rangle_{L^{2}(\mathbb R^{n})}\nonumber \\
&=&\int_{\mathbb R^{n}}\frac{t^{\sigma}}{(t^{2}+|x-z|^{2})^{(n+\sigma)/2}}
\frac{s^{\sigma}}{(s^{2}+|y-z|^{2})^{(n+\sigma)/2}}dz.
\end{eqnarray}

Once we prove the following theorem, then the solution of (\ref{eq-1.2}) will follow from the scheme set in the end of \S 2.

\begin{theorem}\label{thm-4.2}
The fractional Poisson kernel $\mathcal{H}$-$H_{K}$ structure satisfies the BVC. As a consequence, POAFD algorithm can be performed in the context to obtain sparse representation of functions in $\mathcal{H}=L^{2}(\mathbb R^{n})$.
\end{theorem}

\begin{proof}
We first compute the norm $\|K_{q}\|_{H_{K}}$. In fact,
\begin{eqnarray*}
\|K_{q}\|^{2}_{H_{K}}&=&\langle K_{q},\ K_{q}\rangle_{H_{K}}\\
&=&K(q,q)\\
&=&\langle p^{\sigma}_{t}(\cdot),\ p^{\sigma}_{t}(\cdot)\rangle\\
&=&\int_{\mathbb R^{n}}\frac{t^{2\sigma}}{(t^{2}+|x-z|^{2})^{n+\sigma}}dz\\
&=&\frac{1}{t^{n}}\int^{\infty}_{0}\frac{u^{n-1}}{(1+u^{2})^{n+\sigma}}du\\
&=&\frac{c_{n,\sigma}}{(2t)^{n}}.
\end{eqnarray*}
Then we define
$$E_{q}=\frac{K_{q}}{\|K_{q}\|_{H_{K}}}=\left(\frac{(2t)^{n}}{c_{n,\sigma}}\right)^{1/2}K_{q}$$
so that $\|E_{q}\|_{H_{K}}=1$.
Now we verify that the BVC property holds, i.e.,
$$\lim_{q\rightarrow\partial E}|\langle u,\ E_{q}\rangle_{H_{K}}|=0.$$
Due to the density of the span of the kernels $K_p$ it is equivalent to show that, for any but fixed $p=s+y$, under the process $q=t+x\rightarrow\partial\mathbb R^{n+1}_+$,
$$\lim_{q\rightarrow\partial E}|\langle K_{p},\ E_{q}\rangle_{H_{K}}|=0.$$

By the Plancherel formula,
\begin{eqnarray*}
\langle K_{p},\ E_{q}\rangle_{H_{K}}&=&\frac{t^{n/2}}{(c_{n,\sigma})^{1/2}}\langle K_{p},\ K_{q}\rangle\\
&=&\frac{t^{n/2}}{(c_{n,\sigma})^{1/2}}\int_{\mathbb R^{n}}\frac{t^{\sigma}}{(t^{2}+|x-z|^{2})^{(n+\sigma)/2}}\frac{s^{\sigma}}{(s^{2}+|y-z|^{2})^{(n+\sigma)/2}}dz\\
&=&\frac{t^{n/2}}{(c_{n,\sigma})^{1/2}}\int_{\mathbb R^{n}}G_{\sigma}(t|\xi|)G_{\sigma}(s|\xi|)e^{i(x-y)\cdot\xi}d\xi\\
&=&M_{1}+M_{2},
\end{eqnarray*}
where
$$\left\{\begin{aligned}
M_{1}&:=\frac{t^{n/2}}{(c_{n,\sigma})^{1/2}}\int_{|\xi|\leq 1}G_{\sigma}(t|\xi|)G_{\sigma}(s|\xi|)e^{i(x-y)\cdot\xi}d\xi;\\
M_{2}&:=\frac{t^{n/2}}{(c_{n,\sigma})^{1/2}}\int_{|\xi|>1}G_{\sigma}(t|\xi|)G_{\sigma}(s|\xi|)e^{i(x-y)\cdot\xi}d\xi.
\end{aligned}\right.$$

By (ii) of Proposition \ref{prop-4.1}, we can get
\begin{eqnarray*}
|M_{1}|&\lesssim&t^{n/2}\int_{|\xi|\leq1}(1+t^{\sigma}|\xi|^{\sigma})|e^{i(x-y)\xi}|d\xi\\
&\lesssim&t^{n/2}(1+t^{\sigma})\int_{|\xi|\leq1}1d\xi\\
&\lesssim&t^{n/2}(1+t^{\sigma}).
\end{eqnarray*}

For $M_{2}$, we deduce from (ii) of Proposition \ref{prop-4.1} that
\begin{eqnarray*}
|M_{2}|&\lesssim&t^{n/2}\int_{|\xi|\geq1}e^{-(t+s)|\xi|}d\xi\\
&\lesssim&t^{n/2}\int_{|\xi|\geq1}\frac{1}{(1+(t+s)|\xi|)^{n+1}}d\xi\\
&\lesssim&\frac{t^{n/2}}{(t+s)^{n+1}}\int_{|\xi|\geq1}\frac{|\xi|^{n-1}}{|\xi|^{n+1}}d|\xi|\\
&\lesssim&\frac{t^{n/2}}{(t+s)^{n+1}}.
\end{eqnarray*}

The estimates for $M_{1}$ and $M_{2}$ imply that
$$|\langle K_{p},\ E_{q}\rangle_{H_{K}}|\lesssim t^{n/2}(1+t^{\sigma})+\frac{t^{n/2}}{(t+s)^{n+1}},$$
which indicates that
$$\lim_{t\rightarrow 0+}|\langle K_{p},\ E_{q}\rangle_{H_{K}}|=0.$$

Next we deal with the limit as $R=\sqrt{t^{2}+|x|^{2}}\rightarrow\infty$. Without loss of generality, assume that $R\geq =4|y|+2s+1$.

{\it Case 1: $t\geq R/2$}. We can get
\begin{eqnarray*}
\langle K_{p},\ E_{q}\rangle_{H_{K}}&=&\frac{t^{n/2}}{(c_{n,\sigma})^{1/2}}\int_{\mathbb R^{n}}\frac{t^{\sigma}}{(t^{2}+|x-z|^{2})^{(n+\sigma)/2}}\frac{s^{\sigma}}{(s^{2}+|y-z|^{2})^{(n+\sigma)/2}}dz\\
&\lesssim&\frac{t^{\sigma}}{t^{n+\sigma}}\int_{\mathbb R^{n}}\frac{s^{\sigma}}{(s^{2}+|y-z|^{2})^{(n+\sigma)/2}}dz\\
&\lesssim&\frac{1}{t^{n/2}}\int_{0}^{\infty}\frac{u^{n-1}}{(1+u^{2})^{(n+\sigma)/2}}du\\
&\lesssim&\frac{1}{R^{n/2}},
\end{eqnarray*}
which gives
$$\lim_{R\rightarrow \infty}|\langle K_{p},\ E_{q}\rangle_{H_{K}}|=0.$$

{\it Case 2: $0<t<R/2$}. We split
\begin{eqnarray*}
\langle K_{p},\ E_{q}\rangle_{H_{K}}&=&\frac{t^{n/2}}{(c_{n,\sigma})^{1/2}}\int_{\mathbb R^{n}}\frac{t^{\sigma}}{(t^{2}+|x-z|^{2})^{(n+\sigma)/2}}\frac{s^{\sigma}}{(s^{2}+|y-z|^{2})^{(n+\sigma)/2}}dz\\
&\simeq&L_{1}+L_{2},
\end{eqnarray*}
where
$$\left\{\begin{aligned}
L_{1}&:=\frac{t^{n/2}}{(c_{n,\sigma})^{1/2}}\int_{|x-z|\geq |x|/2}\frac{t^{\sigma}}{(t^{2}+|x-z|^{2})^{(n+\sigma)/2}}\frac{s^{\sigma}}{(s^{2}+|y-z|^{2})^{(n+\sigma)/2}}dz;\\
L_{2}&:=\frac{t^{n/2}}{(c_{n,\sigma})^{1/2}}\int_{|x-z|\leq |x|/2}\frac{t^{\sigma}}{(t^{2}+|x-z|^{2})^{(n+\sigma)/2}}\frac{s^{\sigma}}{(s^{2}+|y-z|^{2})^{(n+\sigma)/2}}dz.
\end{aligned}\right.$$

Notice that $0<t<\frac{1}{2}\sqrt{t^{2}+|x|^{2}}$. We can see that $|x|>\sqrt{3}t$, i.e., $|x|>\sqrt{3}R/2$. For $L_{1}$, since $|x-z|>|x|/2>\sqrt{3}R/4$,
we obtain
\begin{eqnarray*}
L_{1}&\lesssim&\frac{t^{n/2+\sigma}}{R^{n+\sigma}}\int_{|x-z|\geq |x|/2}\frac{s^{\sigma}}{(s^{2}+|y-z|^{2})^{(n+\sigma)/2}}dz\\
&\lesssim&\frac{R^{n/2+\sigma}}{R^{n+\sigma}}\int_{\mathbb R^{n}}\frac{s^{\sigma}}{(s^{2}+|y-z|^{2})^{(n+\sigma)/2}}dz\\
&\lesssim&\frac{1}{R^{n/2}}.
\end{eqnarray*}

For $L_{2}$, since $|x-z|<|x|/2$ and $R>4|y|$, then by the triangle inequality, we have
\begin{eqnarray*}
|y-z|&\geq&|z|-|y|\geq|x|-|x-z|-|y|\\
&\geq&|x|-|x|/2-|y|\\
&\geq&(\sqrt{3}-1)R/4.
\end{eqnarray*}
Then
\begin{eqnarray*}
L_{2}&\lesssim&\int_{|x-z|\leq |x|/2}\frac{t^{\sigma+n/2}}{(t^{2}+|x-z|^{2})^{(n+\sigma)/2}}\frac{s^{\sigma}}{(s^{2}+|y-z|^{2})^{(n+\sigma)/2}}dz\\
&\lesssim&\int_{|x-z|\leq |x|/2}\frac{1}{|y-z|^{n}}\frac{t^{\sigma+n/2}}{(t^{2}+|x-z|^{2})^{(n+\sigma)/2}}dz\\
&\lesssim&\frac{t^{n/2}}{((\sqrt{3}-1)R/4)^{n}}\int_{\mathbb R^{n}}\frac{t^{\sigma}}{(t^{2}+|x-z|^{2})^{(n+\sigma)/2}}dz\\
&\lesssim&\frac{1}{R^{n/2}}.
\end{eqnarray*}

The estimates for $I_{1}$ and $I_{2}$ indicate that
$$\lim_{R\rightarrow \infty}|\langle K_{p},\ E_{q}\rangle_{H_{K}}|=0.$$
This completes the proof of Theorem \ref{thm-4.2}.
\end{proof}

\section{Experiments}
In this section, we present two examples to illustrate the effectiveness and validity of POAFD method. The decomposition results are shown in figure \ref{figure1} and \ref{figure3}. In figure \ref{figure2} and \ref{figure4}, we give the corresponding parameters and relative error.

\begin{example} This example is for the problem (\ref{eq-1.1}) with the initial value
$$f=\frac{0.5}{0.25+(0.5-x)^2}+\frac{2\pi}{1+(0.5+x)^2}+\frac{0.8}{0.64+(1+x)^2}.$$
For a fixed $\alpha$ we use the dictionary $K_q(\cdot)=K_{\alpha,t}(x-\cdot), q=t+x\in \mathbb R^{n+1}_+,$ where $K_{\alpha,t}(x)$ is referred to (\ref{ref}).

Let $K^\top=(K_{q_1},\ldots,K_{q_{13}})$ and $E^\top=(E_{1},\ldots,E_{13})$. We have
$K=A^\top E,$ where
\begin{equation*}\label{matrix}
A=\left(\begin{smallmatrix}
    1     & 0.37  & 0.99  & 0.49  & 0.75  & 0.79  & 0.30  & 0.91  & 0.42  & 0.66  & 0.65  & 0.23  & 0.42  & 0.91  \\
    0     & 0.93  & -0.06  & 0.84  & 0.31  & 0.00  & 0.50  & 0.14  & 0.18  & -0.04  & 0.22  & 0.86  & 0.59  & 0.03  \\
    0     & 0     & 0.14  & -0.11  & -0.34  & -0.34  & -0.12  & -0.24  & -0.35  & -0.32  & -0.34  & 0.08  & -0.16  & -0.26  \\
    0     & 0     & 0     & 0.20  & 0.16  & -0.12  & 0.49  & 0.03  & 0.19  & -0.16  & 0.12  & -0.14  & 0.46  & -0.07  \\
    0     & 0     & 0     & 0     & 0.45  & -0.26  & -0.18  & 0.19  & -0.18  & -0.29  & 0.62  & 0.00  & -0.16  & -0.19  \\
    0     & 0     & 0     & 0     & 0     & 0.42  & -0.03  & -0.14  & -0.11  & 0.56  & 0.02  & 0.01  & -0.03  & 0.22  \\
    0     & 0     & 0     & 0     & 0     & 0     & 0.61  & 0.02  & -0.08  & 0.01  & 0.03  & -0.10  & 0.44  & 0.00  \\
    0     & 0     & 0     & 0     & 0     & 0     & 0     & 0.17  & -0.16  & 0.08  & -0.02  & 0.01  & -0.01  & -0.09  \\
    0     & 0     & 0     & 0     & 0     & 0     & 0     & 0     & 0.74  & -0.02  & -0.01  & -0.01  & 0.05  & -0.01  \\
    0     & 0     & 0     & 0     & 0     & 0     & 0     & 0     & 0     & 0.16  & 0.01  & 0.00  & -0.01  & -0.01  \\
    0     & 0     & 0     & 0     & 0     & 0     & 0     & 0     & 0     & 0     & 0.12  & 0.03  & 0.02  & 0.01  \\
    0     & 0     & 0     & 0     & 0     & 0     & 0     & 0     & 0     & 0     & 0     & 0.42  & -0.04  & 0 \\
    0     & 0     & 0     & 0     & 0     & 0     & 0     & 0     & 0     & 0     & 0     & 0     & 0.10  & 0 \\
    0     & 0     & 0     & 0     & 0     & 0     & 0     & 0     & 0     & 0     & 0     & 0     & 0     & 0.10
\end{smallmatrix}\right)_{14\times14}.
\end{equation*}
$$\Big\|f-\sum_{k=1}^{13}c_k\widetilde{K}_{q_k}\Big\|\leq 1.1668\times10^{-4}.$$
The isometric relation then gives
\[ \Big\|u_f-\sum_{k=1}^{13}c_k\widetilde{K}_{q_k}\Big\|_{H_K}\leq 1.1668\times10^{-4},\]
where the coefficient $c_k$ and $\langle f,E_k\rangle$ are given in table \ref{table1}.
\begin{table}[htbp]
\myfont
  \centering
  \begin{tabular}{|c|c|c|c|c|c|c|c|c|c|c|c|c|c|c|}
    \toprule
    $c_k$     & -33.8652  & 0.7736  & -6.1651  & 0.0093  & -0.2163  & -0.2053  & -0.1342  & 0.1869  & -0.0153  & -0.0916  & -0.0593  & 0.0238  & 0.0642  & 0.0551  \\
    \midrule
    $\langle f,E_k\rangle$     & 0.5183  & -0.2303  & -0.1278  & -0.0515  & 0.0462  & 0.0254  & 0.0249  & 0.0187  & -0.0128  & -0.0139  & -0.0102  & 0.0072  & 0.0066  & 0.0056  \\
    \bottomrule
    \end{tabular}%
  \caption{The coefficients of $K_{q_k}$ and $E_k$ }
    \label{table1}%
\end{table}%

\end{example}

\begin{figure}[H]
	\begin{minipage}[c]{0.23\textwidth}
		\centering
		\includegraphics[height=3.5cm,width=3cm]{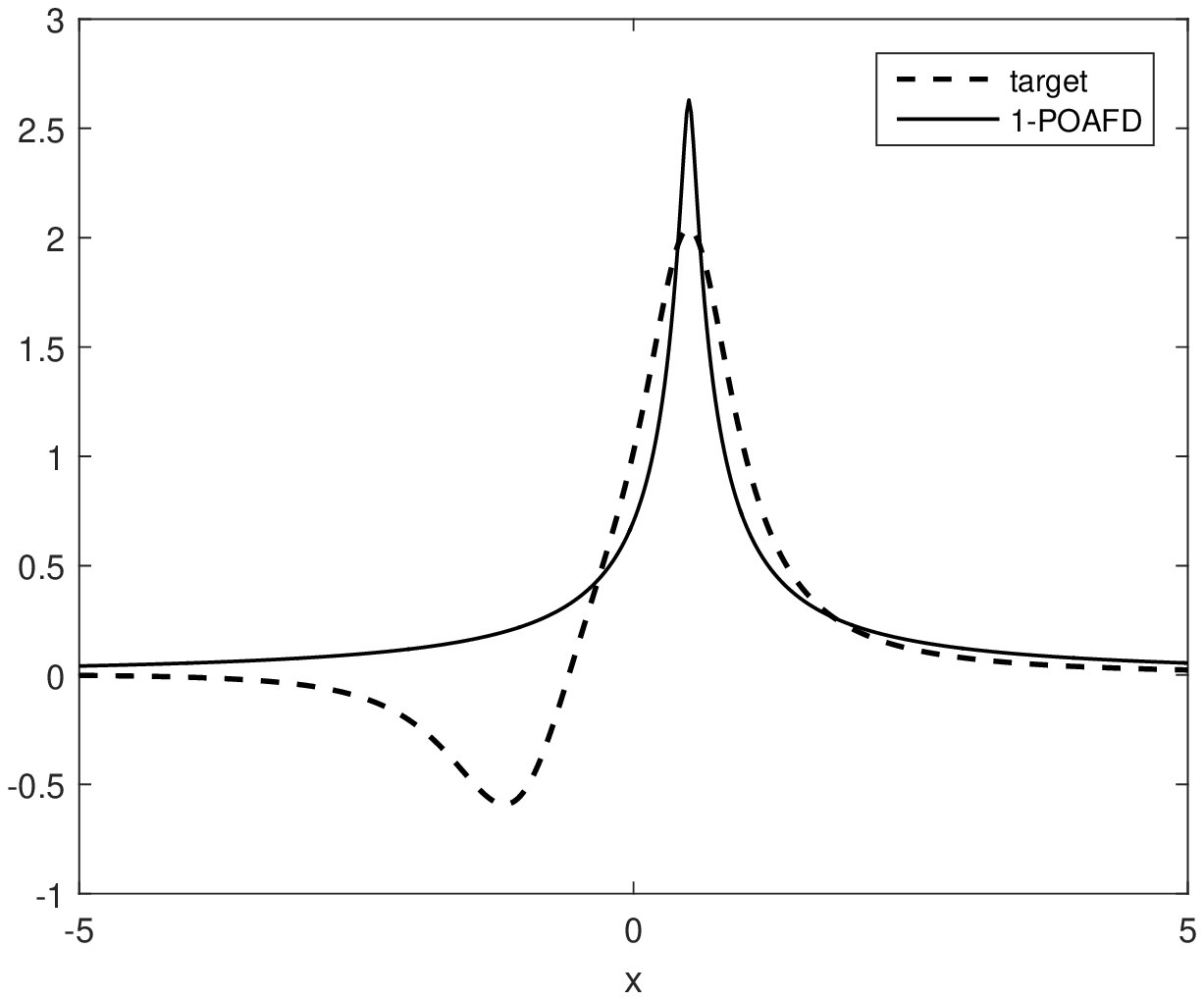}
		\centerline{{\scriptsize POAFD: 1 partial sum}}
	\end{minipage}
	\begin{minipage}[c]{0.23\textwidth}
		\centering
		\includegraphics[height=3.5cm,width=3cm]{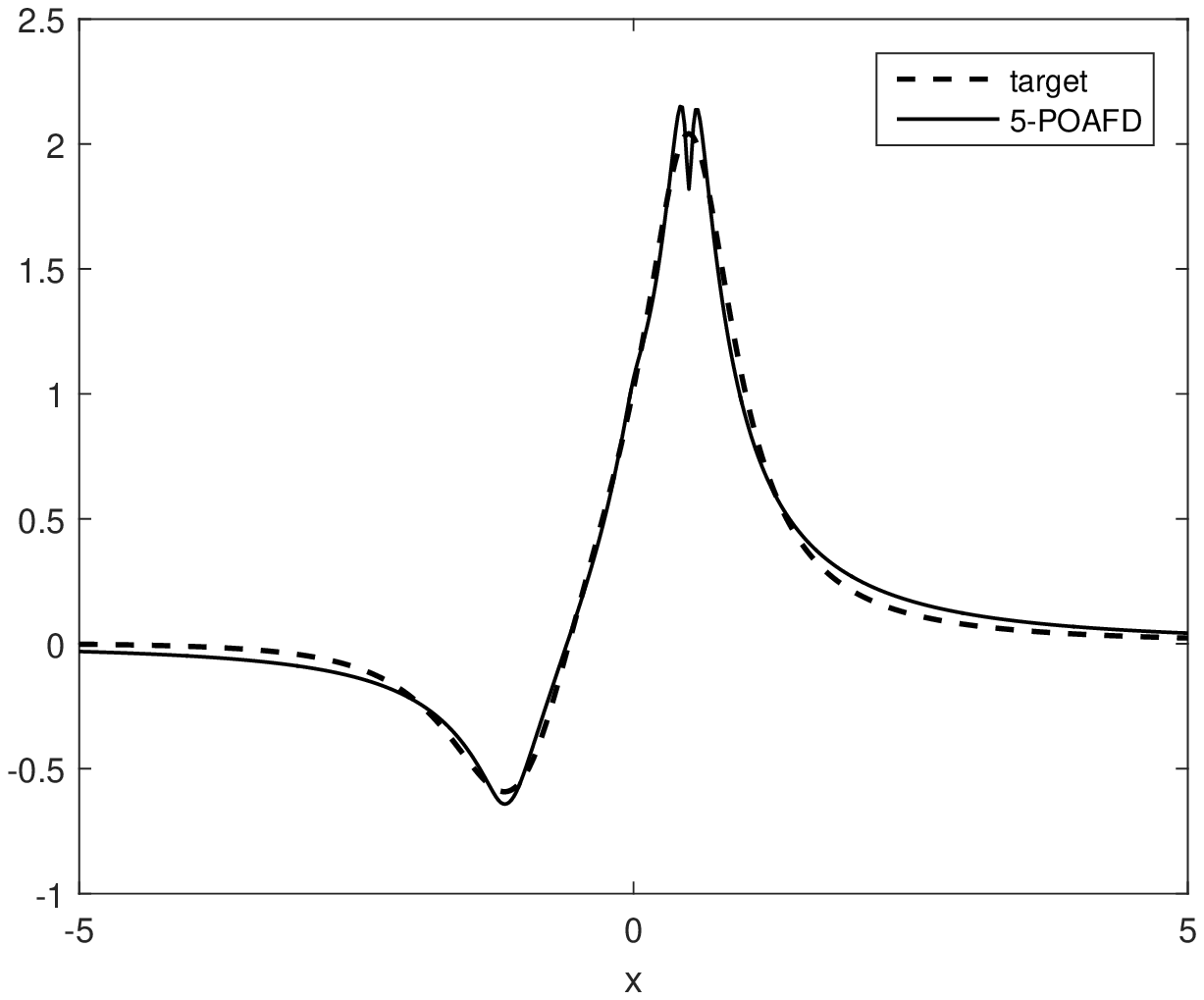}
		\centerline{{\scriptsize POAFD: 5 partial sum}}
	\end{minipage}
	\begin{minipage}[c]{0.23\textwidth}
		\centering
		\includegraphics[height=3.5cm,width=3cm]{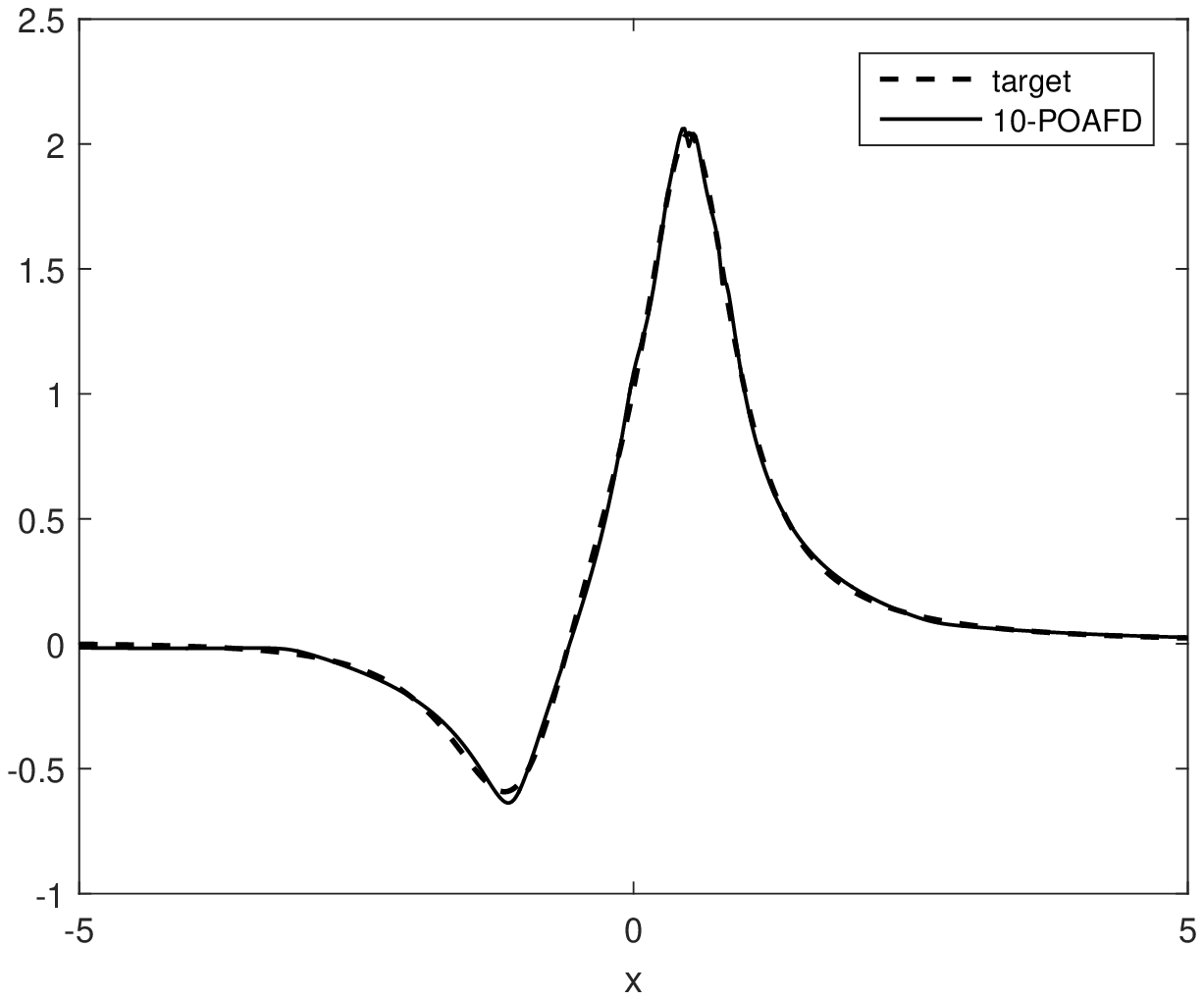}
		\centerline{{\scriptsize POAFD: 10 partial sum}}
	\end{minipage}
	\begin{minipage}[c]{0.23\textwidth}
		\centering
		\includegraphics[height=3.5cm,width=3cm]{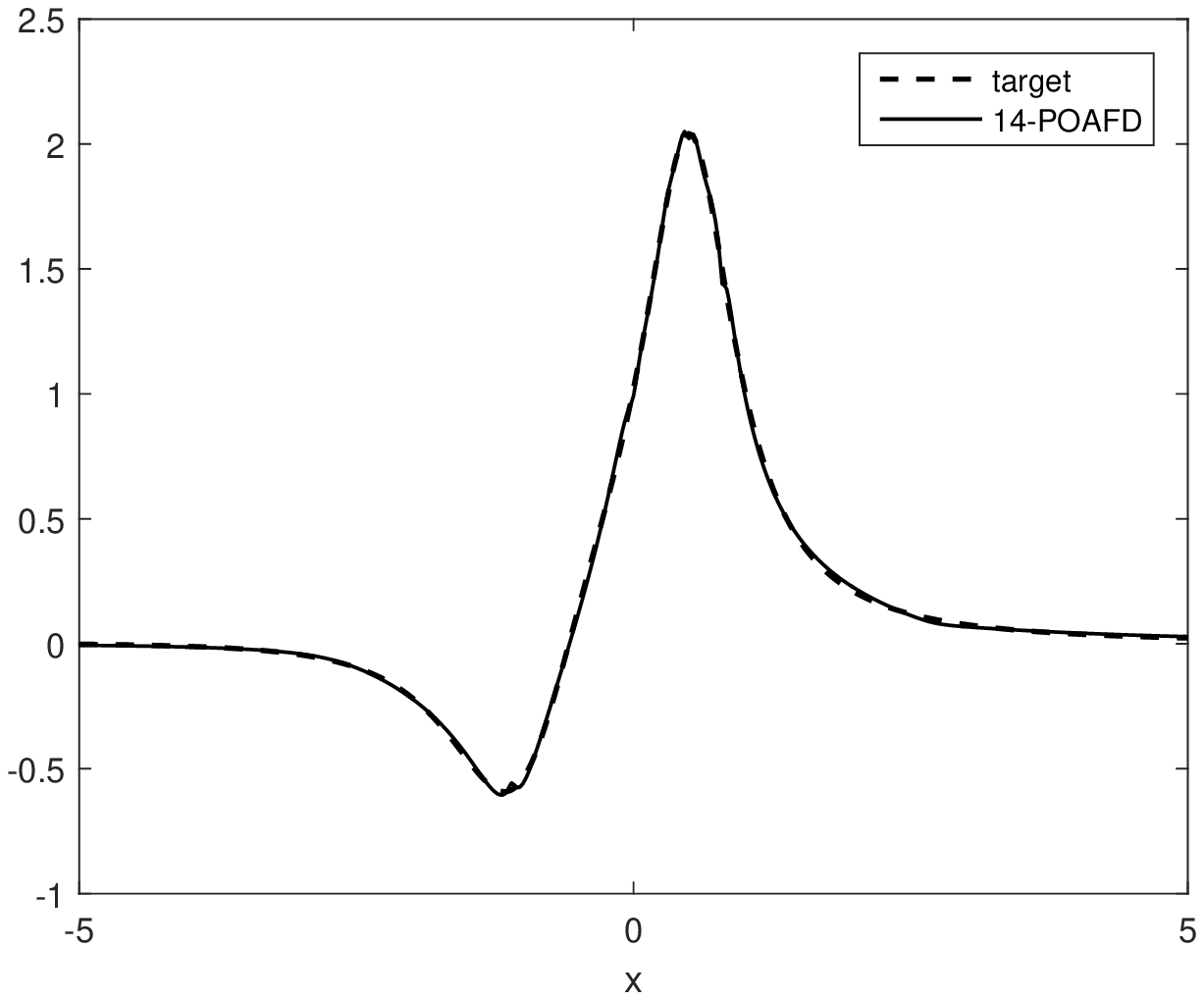}
		\centerline{{\scriptsize POAFD: 14 partial sum}}
	\end{minipage}
	\caption{\scriptsize Approximate $f$ by POAFD with heat integral kernel on the upper half plane}
\label{figure1}
\end{figure}

\begin{figure}[H]
	\begin{minipage}[c]{0.45\textwidth}
		\centering
		\includegraphics[height=5cm,width=6cm]{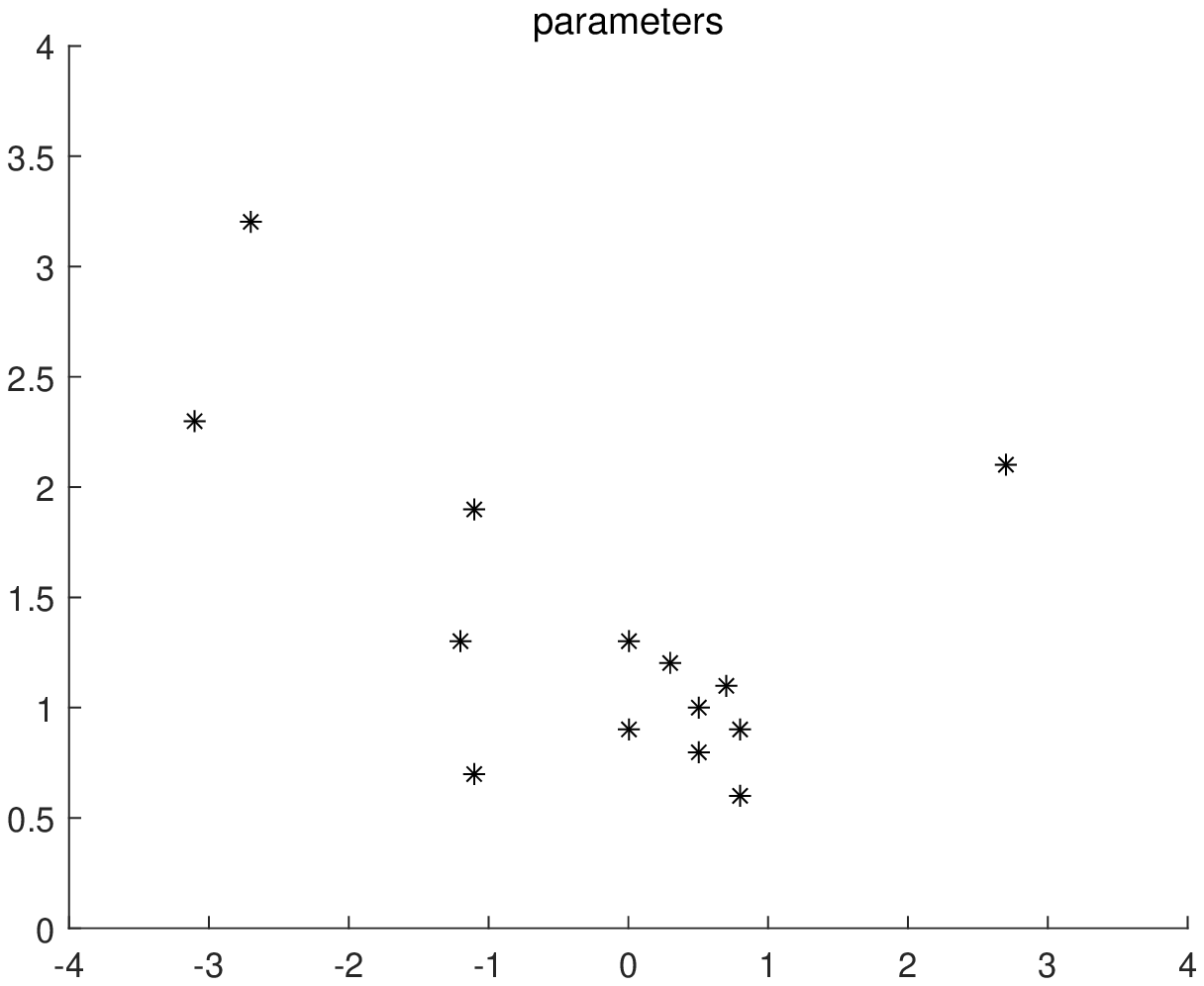}
		\centerline{{\scriptsize parameters}}
	\end{minipage}
	\begin{minipage}[c]{0.45\textwidth}
		\centering
		\includegraphics[height=5cm,width=6cm]{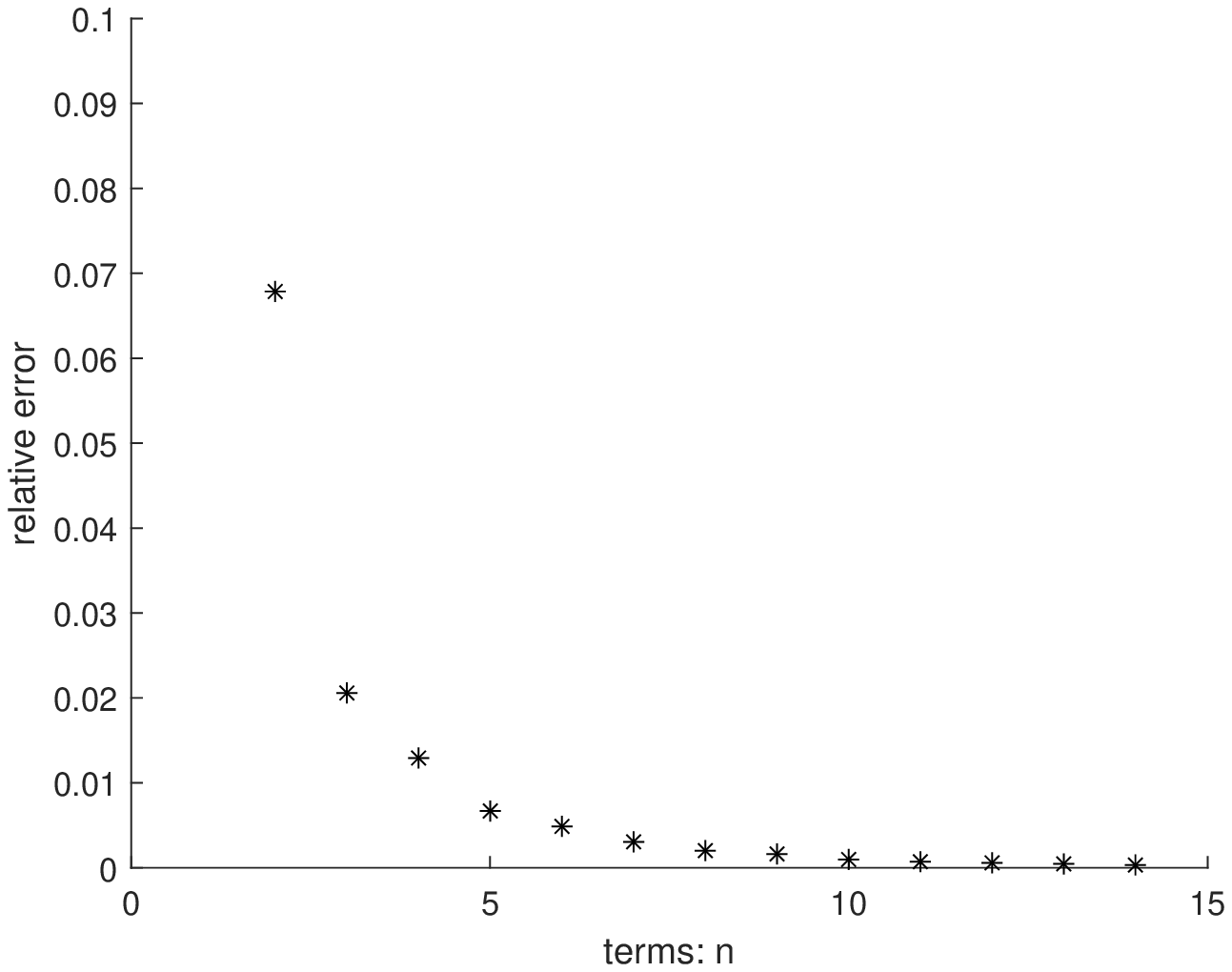}
		\centerline{{\scriptsize relative error}}
	\end{minipage}
	\caption{\scriptsize parameters and relative error}
\label{figure2}
\end{figure}

\begin{example} The next example is for the problem (\ref{eq-1.2}) with the initial value
$$f=\sum_{j=1}^5a_je^{-\frac{|b_j-x|^2}{a_j}}\cos x,$$
where $a_j=0.5,3.1,2.4,0.2,1.6$ and $b_j=1,0.6,-0.04,2.3,-2.$ Let now the dictionary consists of
$K_q(y)=h_p(y)$ given by (\ref{givenby}).

Let $K^\top=(K_{q_1},\ldots,K_{q_{13}})$ and $E^\top=(E_{1},\ldots,E_{13}).$ We have
$K=A^\top E,$ where
\begin{equation*}\label{matrix}
A=\left(\begin{smallmatrix}
    1     & 0.24  & 0.26  & 0.99  & 0.34  & 0.21  & 0.92  & 0.33  & 0.65  & 0.44  & 0.24  & 0.83  & 0.14  & 0.50  & 0.08  \\
    0     & 0.97  & 0.05  & -0.04  & 0.07  & 0.57  & 0.16  & 0.92  & 0.13  & 0.52  & 0.04  & -0.01  & 0.52  & 0.02  & 0.03  \\
    0     & 0     & 0.97  & -0.04  & 0.92  & -0.01  & 0.20  & 0.03  & 0.60  & 0.02  & 0.91  & -0.07  & 0.02  & 0.30  & 0.25  \\
    0     & 0     & 0     & 0.12  & -0.07  & -0.11  & -0.25  & -0.07  & -0.26  & -0.28  & -0.02  & -0.11  & -0.02  & -0.38  & -0.04  \\
    0     & 0     & 0     & 0     & 0.17  & -0.03  & 0.09  & 0.00  & 0.32  & -0.07  & 0.13  & -0.14  & 0.02  & 0.38  & 0.28  \\
    0     & 0     & 0     & 0     & 0     & 0.79  & 0.03  & 0.12  & 0.00  & 0.45  & 0.00  & -0.02  & -0.13  & -0.05  & 0.00  \\
    0     & 0     & 0     & 0     & 0     & 0     & 0.10  & 0.06  & 0.12  & 0.23  & -0.06  & -0.25  & 0.18  & 0.32  & 0.00  \\
    0     & 0     & 0     & 0     & 0     & 0     & 0     & 0.12  & -0.06  & 0.18  & 0.02  & 0.13  & 0.62  & -0.27  & 0.06  \\
    0     & 0     & 0     & 0     & 0     & 0     & 0     & 0     & 0.10  & -0.26  & -0.05  & -0.06  & 0.24  & 0.36  & -0.07  \\
    0     & 0     & 0     & 0     & 0     & 0     & 0     & 0     & 0     & 0.31  & 0.00  & 0.08  & -0.24  & -0.02  & -0.05  \\
    0     & 0     & 0     & 0     & 0     & 0     & 0     & 0     & 0     & 0     & 0.29  & -0.04  & 0.03  & 0.09  & -0.35  \\
    0     & 0     & 0     & 0     & 0     & 0     & 0     & 0     & 0     & 0     & 0     & 0.42  & -0.03  & 0.07  & 0.02  \\
    0     & 0     & 0     & 0     & 0     & 0     & 0     & 0     & 0     & 0     & 0     & 0     & 0.40  & -0.01  & -0.05  \\
    0     & 0     & 0     & 0     & 0     & 0     & 0     & 0     & 0     & 0     & 0     & 0     & 0     & 0.22  & -0.29  \\
    0     & 0     & 0     & 0     & 0     & 0     & 0     & 0     & 0     & 0     & 0     & 0     & 0     & 0     & 0.79
\end{smallmatrix}\right)_{15\times15}.
\end{equation*}
$$\Big\|f-\sum_{k=1}^{13}c_k\widetilde{K}_{q_k}\Big\|\leq 1.3486\times10^{-4}.$$
We also conclude
\[\Big\|u_f-\sum_{k=1}^{13}c_k\widetilde{K}_{q_k}\Big\|_{H_K}\leq 1.3486\times10^{-4},\]
where the coefficient $c_k$ and $\langle f,E_k\rangle$ are given in table \ref{table2}.
\begin{table}[htbp]
\myfont
  \centering
    \begin{tabular}{|c|c|c|c|c|c|c|c|c|c|c|c|c|c|c|c|}
    \toprule
    $c_k$ & 18.3  & -1.39 & 1.11  & -0.64 & -0.49 & -0.04 & -3.03 & 0.26  & 0.12  & -0.39 & 0.09  & 0.04  & 0.06  & 0.07  & 0.01 \\
    \midrule
   $\langle f,E_k\rangle$& 1.85  & -0.6  & -0.5  & -0.51 & -0.15 & -0.13 & 0.16  & -0.09 & -0.08 & -0.07 & 0.03  & 0.03  & 0.02  & 0.01  & 0.01 \\
    \bottomrule
    \end{tabular}%
  \caption{The coefficients of $K_{q_k}$ and $E_k$ }
    \label{table2}%
\end{table}%

\end{example}

\begin{figure}[H]
	\begin{minipage}[c]{0.23\textwidth}
		\centering
		\includegraphics[height=3.5cm,width=3.2cm]{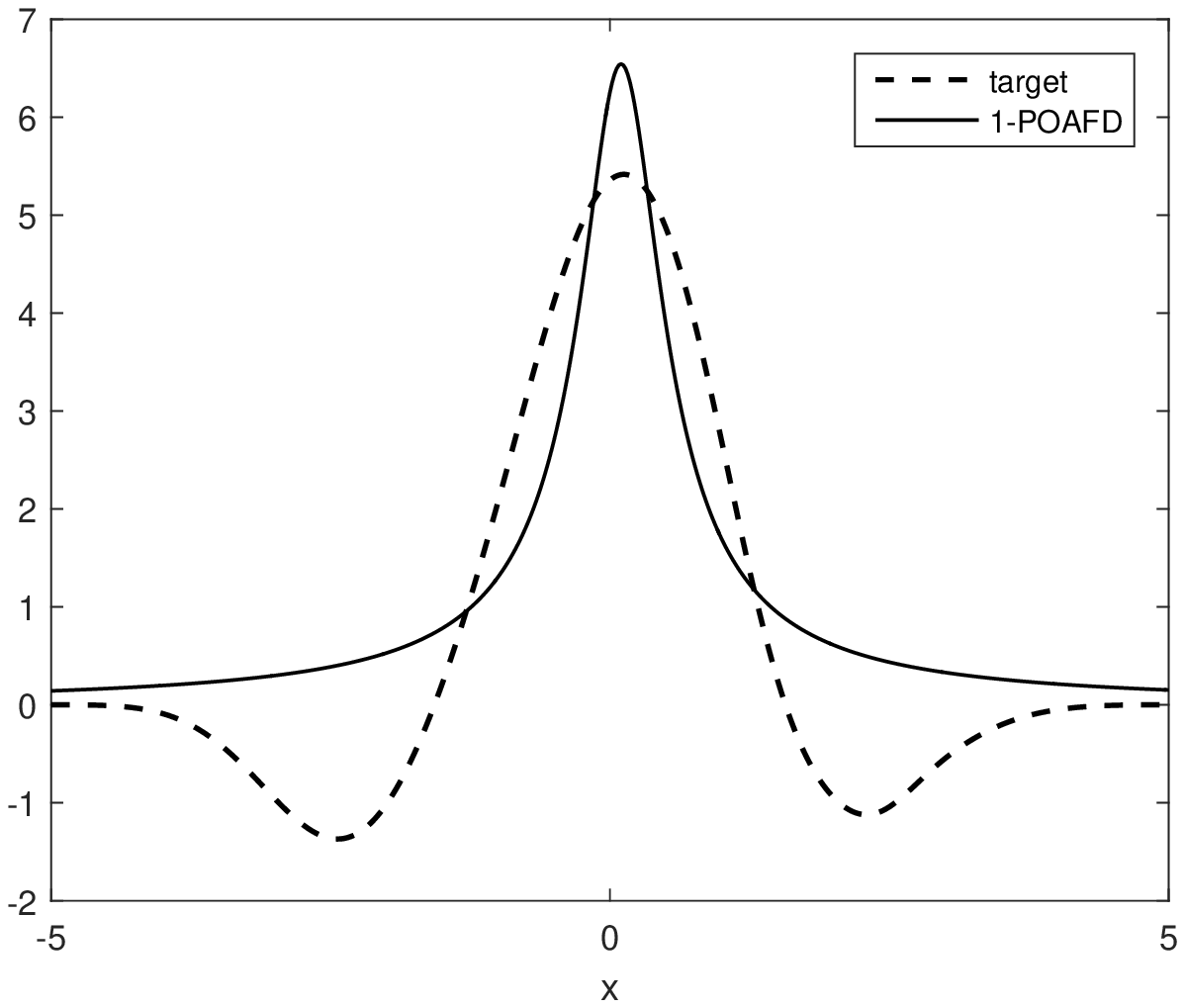}
		\centerline{{\scriptsize POAFD: 1 partial sum}}
	\end{minipage}
	\begin{minipage}[c]{0.23\textwidth}
		\centering
		\includegraphics[height=3.5cm,width=3.2cm]{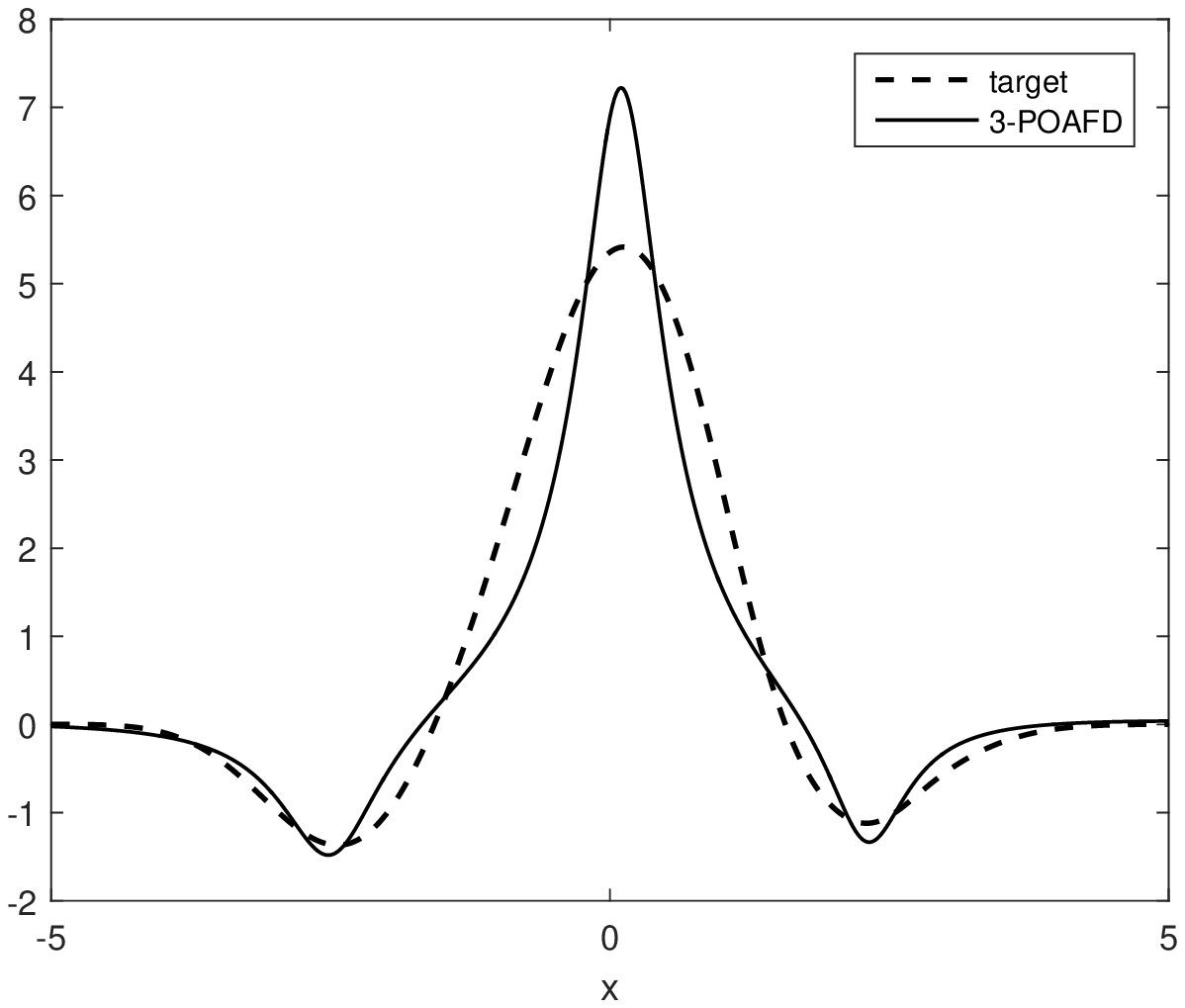}
		\centerline{{\scriptsize POAFD: 3 partial sum}}
	\end{minipage}
	\begin{minipage}[c]{0.23\textwidth}
		\centering
		\includegraphics[height=3.5cm,width=3.2cm]{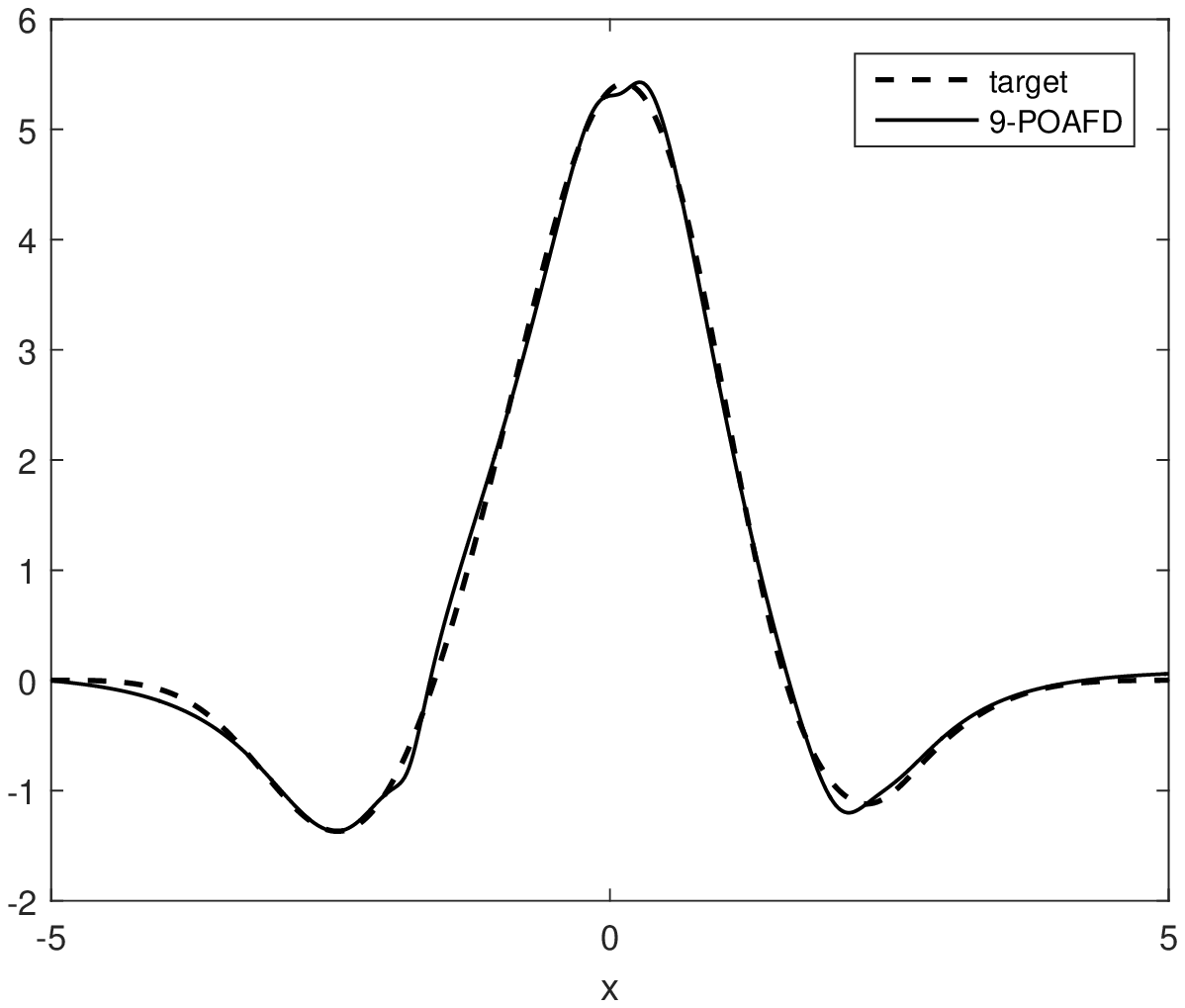}
		\centerline{{\scriptsize POAFD: 9 partial sum}}
	\end{minipage}
	\begin{minipage}[c]{0.23\textwidth}
		\centering
		\includegraphics[height=3.5cm,width=3.2cm]{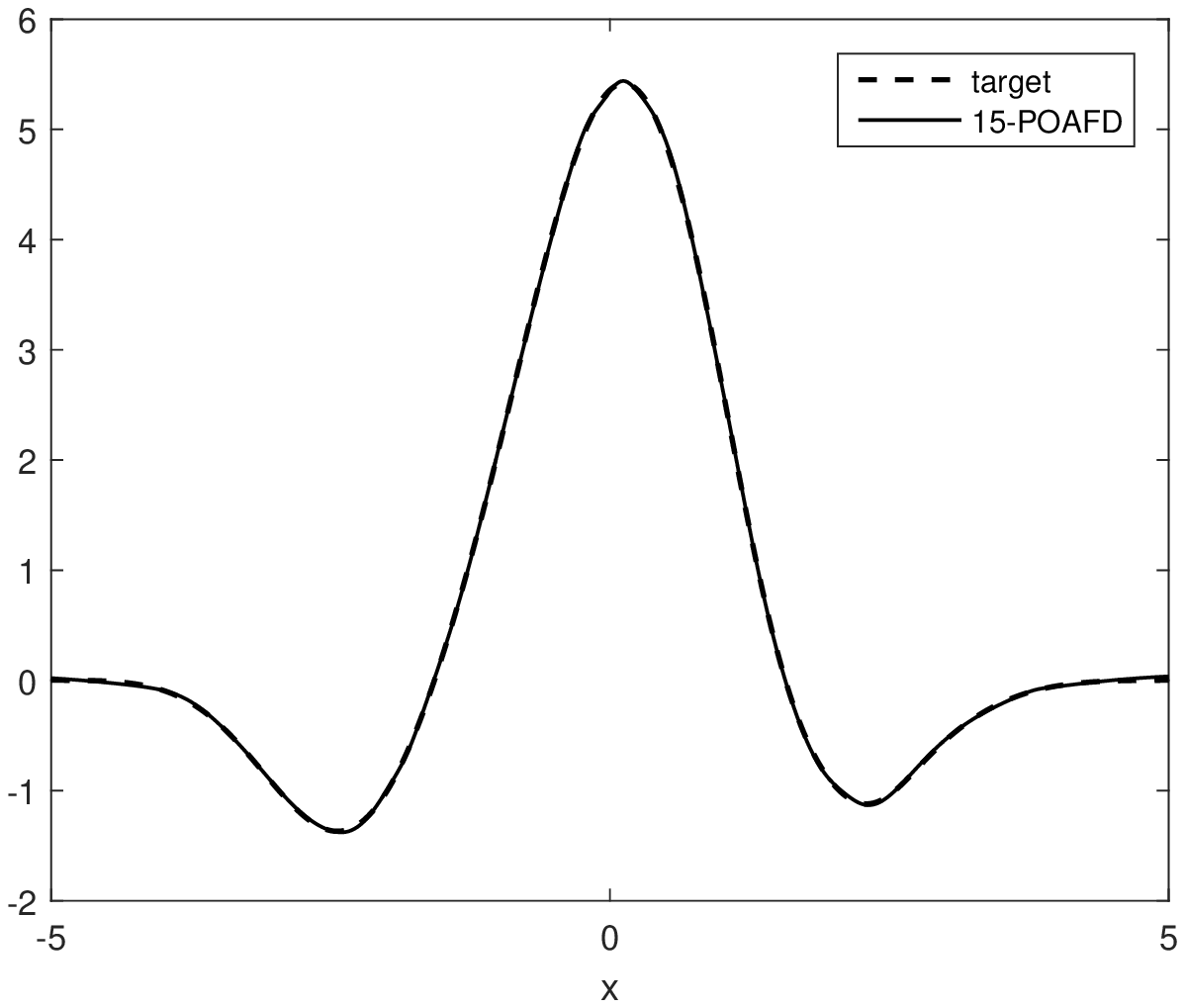}
		\centerline{{\scriptsize POAFD: 15 partial sum}}
	\end{minipage}
	\caption{\scriptsize Approximate $f$ by POAFD with Poisson kernel on the upper half plane}
\label{figure3}
\end{figure}

\begin{figure}[H]
	\begin{minipage}[c]{0.45\textwidth}
		\centering
		\includegraphics[height=5cm,width=6cm]{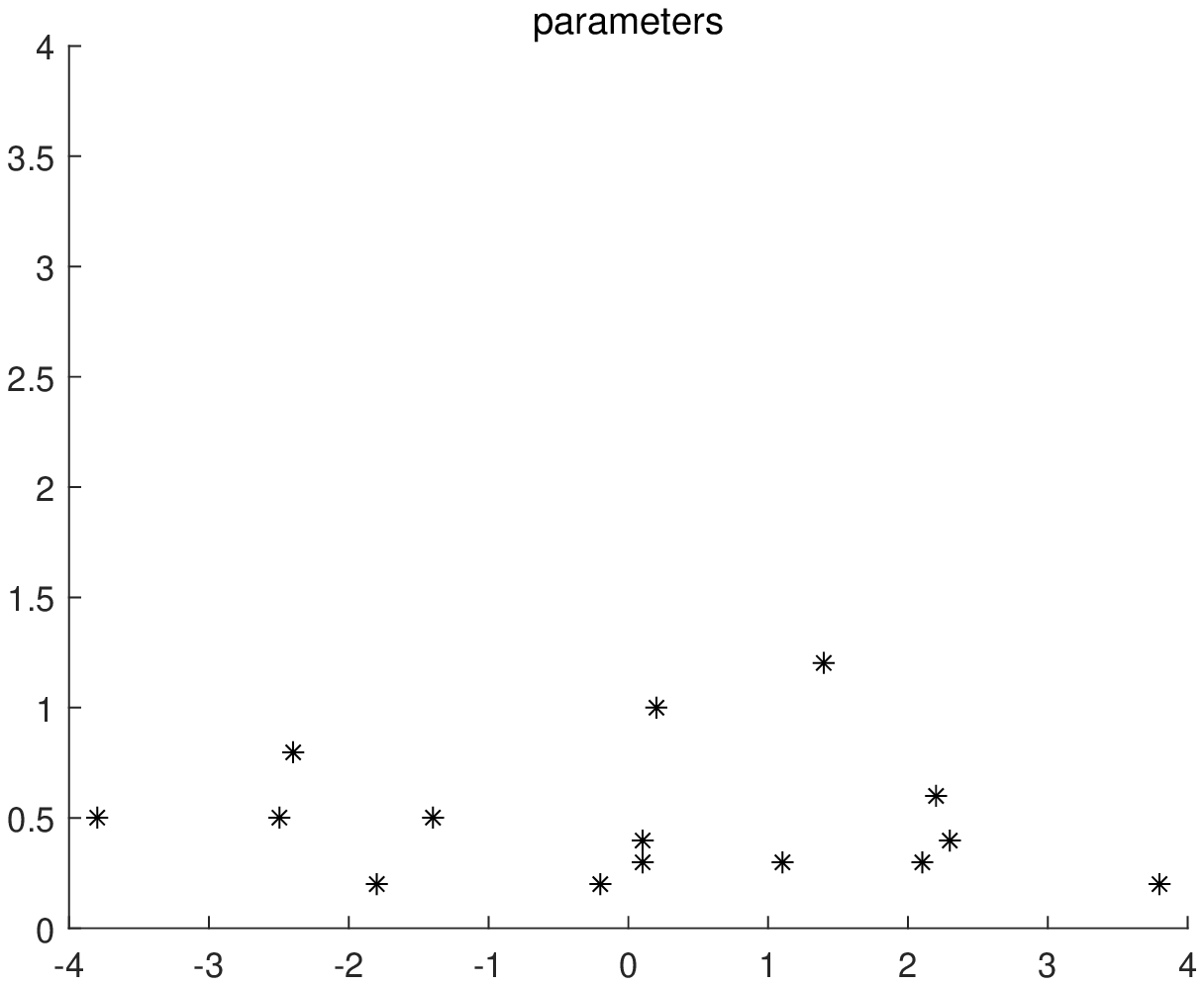}
		\centerline{{\scriptsize parameters}}
	\end{minipage}
	\begin{minipage}[c]{0.45\textwidth}
		\centering
		\includegraphics[height=5cm,width=6cm]{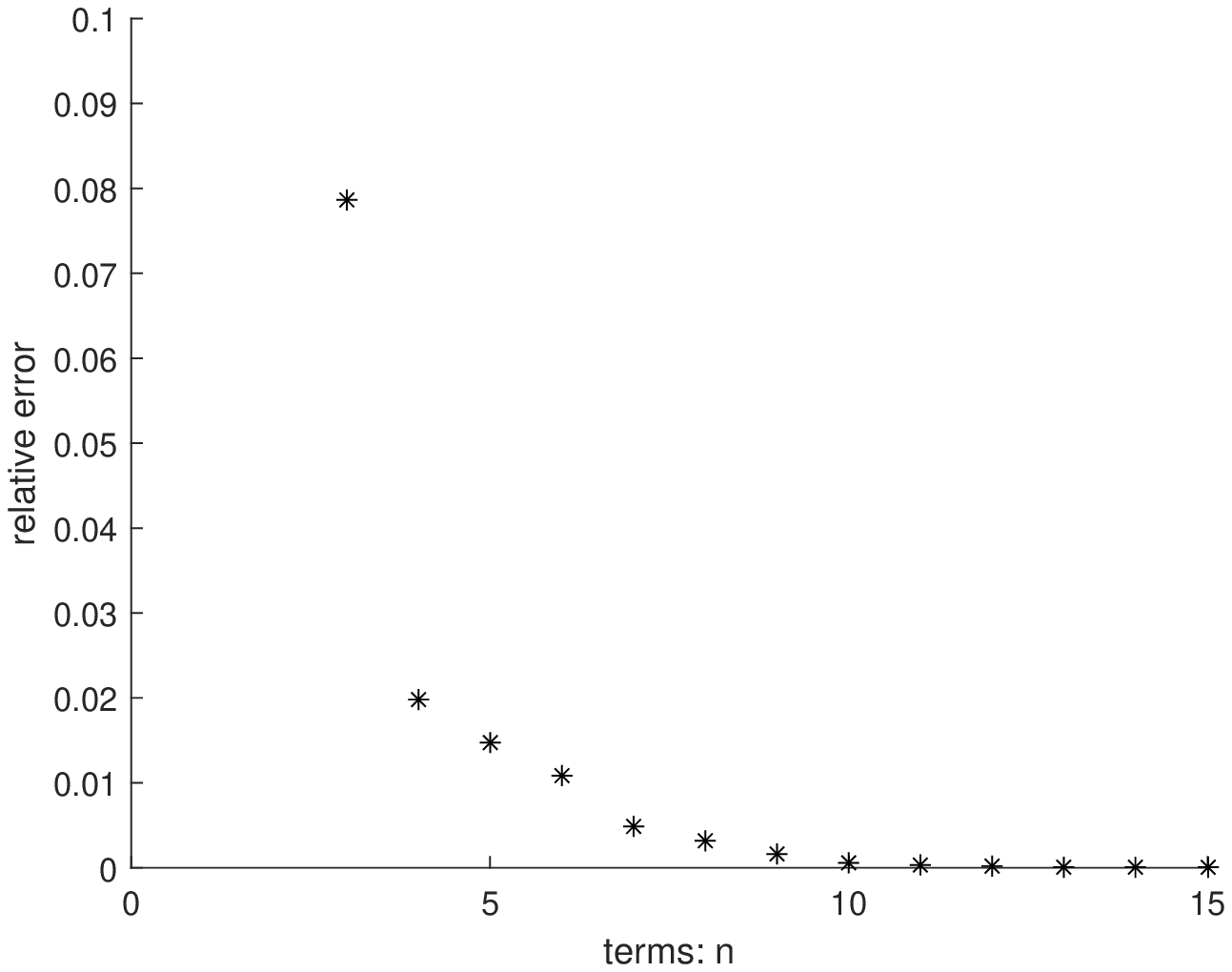}
		\centerline{{\scriptsize relative error}}
	\end{minipage}
	\caption{\scriptsize parameters and relative error}
\label{figure4}
\end{figure}

 \vspace{0.1in} \noindent

 {\bf{Acknowledgements.}}
Tao Qian is supported by the Science and Technology Development Fund, Macau SAR (File no. 0123/2018/A3). Pengtao Li is supported by the National Natural Science Foundation of China (Grant Nos. 12071272, 11871293) and the Shandong Natural Science Foundation of China (Grant No. ZR2020MA004). Ieng Tak Leong is supported by MYRG 2018-00168-FST.

\tabcolsep 30pt

\vspace*{7pt}

\end{document}